\documentclass[12pt]{article}

\usepackage
[
a4paper,
margin=3cm
]
{geometry}
\usepackage[T1]{fontenc}
\usepackage{scrextend}
\usepackage{multicol}

\usepackage{amsmath,amssymb,enumerate,amsbsy,amsthm,parskip,xcolor,enumitem,cancel,tikz,tikz-cd,graphicx,adjustbox,fancyhdr,comment}
\usepackage{appendix,setspace}
\usepackage[hidelinks]{hyperref}

\renewcommand{\sectionmark}[1]{}

\newcommand{\mbr}{\mathbb{R}}
\newcommand{\mbz}{\mathbb{Z}}
\newcommand{\mbc}{\mathbb{C}}
\newcommand{\mbcp}{\mathbb{CP}}
\newcommand{\mbh}{\mathbb{H}}

\DeclareMathOperator{\Ad}{\mathrm{Ad}}
\newcommand{\mca}{\mathcal{A}}
\newcommand{\mfh}{\mathfrak{h}}
\newcommand{\mfg}{\mathfrak{g}}
\newcommand{\mfk}{\mathfrak{k}}
\newcommand{\mfm}{\mathfrak{m}}
\newcommand{\mrd}{\mathrm{d}}
\newcommand{\del}{\partial}

\newcommand{\mr}[1]{\mathrm{#1}}
\newcommand{\on}[1]{\operatorname{#1}}

\setlist[enumerate,1]{label=(\arabic*)}

\renewenvironment{align*}
{\[
	\begin{aligned}}
	{\end{aligned}
\]}

\theoremstyle{definition}
\newtheorem{theorem}{Theorem}[section]
\newtheorem{proposition}[theorem]{Proposition}
\newtheorem{lemma}[theorem]{Lemma}
\newtheorem{corollary}[theorem]{Corollary}
\newtheorem{definition}[theorem]{Definition}
\newtheorem{example}[theorem]{Example}

\newtheorem{remark}[theorem]{Remark}

\title{\singlespacing\textbf{Invariant Complex Structures \\
		and Kähler Metrics on \\
		Homogeneous Principal Bundles}}
\date{}
\author{Eric Fung Ya Ho}

\begin{document}
	\maketitle
	\begin{abstract}
		
		This work investigates the existence and classification of almost complex structures and Kähler metrics on principal bundles, with particular emphasis on the homogeneous setting. Using Wang's theory of invariant connections, we give a direct geometric proof of the classification of invariant integrable complex structures established by Biswas and Upmeier \cite{biswas}. This approach avoids the algebraic machinery of Jordan triple systems and extends the classification from Hermitian symmetric spaces to general symmetric spaces admitting invariant complex structures. We also present a computationally direct proof of Johnson's Kähler criterion \cite{johnson1980kahler} using the Levi-Civita connection. Finally, we apply these classification results to the reduced frame bundles of the upper half-plane and complex projective spaces, showing that the invariant integrable complex structures in these cases are unique.
		
		\vspace{2ex}
	\end{abstract}
	\tableofcontents
	\newpage

	\section{Introduction}
	
	\subsection*{Background and Motivation}
	In differential geometry, an almost complex structure $J$ on a smooth manifold is an endomorphism on the tangent bundle such that $J^2=-\on{id}$. Such structures generalise the notion of complex manifolds, allowing mathematicians to study manifolds that are not complex in the integrable sense, but still exhibit complex-like behaviour at the infinitesimal level. On the other hand, a principal $G$-bundle provides the natural geometric framework for describing connections, curvatures, and gauge symmetries. The interaction between the theory of almost complex structures and the study of principal $G$-bundles leads to interesting questions: under what conditions can the total space of a principal bundle carry an almost complex structure, and how do the choices of connections affect the nature of such structures?
	
	The study of almost complex structures has been considered from various perspectives. In 1978, in their seminal work \cite{atiyah1978self}, Atiyah, Hitchin, and Singer showed that an oriented Riemannian 4-manifold is self-dual if and only if the almost complex structure on its associated twistor space, induced by the Riemannian metric, is integrable. Since twistor spaces arise as total spaces of certain principal bundles, this result motivated the study of almost complex structures on principal bundles. In this context, in the paper \cite{biswas}, Biswas and Harald classified invariant and integrable almost complex structures on homogeneous principal bundles in terms of Lie-algebraic data. Independently, Johnson formulated a condition for a principal bundle to be K\"ahler \cite{johnson1980kahler}.
	
	\subsection*{Summary of Work}
	This work investigates the existence and classification of almost complex structures and K\"ahler metrics on the total space of a principal bundle. Special attention is given to homogeneous principal bundles, where the global symmetry allows for a Lie-theoretic description of these structures. A central part of our investigation concerns invariant connections, almost complex structures, and the K\"ahler metrics they induce on a homogeneous principal bundle. In this setting, the symmetry enables a precise classification of these structures in terms of Lie-algebraic data.
	
	Our contributions are as follows:
	
	\begin{itemize} 
		\item \textbf{Alternative Proof for Integrability in the Symmetric Case:} We provide an alternative proof of the classification of invariant integrable complex structures established in \cite{biswas} (Theorem \ref{classification of integrable invariant complexions}). While the original result for Hermitian symmetric spaces relies heavily on Jordan triple systems, we provide a direct, alternative proof using Wang's description of invariant connections. Moreover, we relax the geometric assumption from Hermitian symmetric spaces to general symmetric spaces with invariant complex structures.
		\item \textbf{K\"ahler Metrics on Principal Bundles:} We revisit Johnson's criterion for a principal bundle to be K\"ahler. Originally proved using the geometry of totally geodesic fibres, we provide an alternative proof using a direct computation based on the Levi-Civita connection (Theorem \ref{kahler iff flat}). \item \textbf{Classification in the Homogeneous Setting:}  We apply Wang's classification theorem to explicitly classify K\"ahler metrics of certain forms on homogeneous principal bundles (Theorem \ref{classification of kahler metrics (Gx_lambda H)}).
		\item \textbf{Applications:} We apply the classification theorems to classify invariant connections and invariant integrable complexions on the reduced frame bundles of the upper half plane and complex projective spaces. In particular, we show that they are unique (Propositions \ref{unique complexion on reduced frame bundle of H} and \ref{unique complexion on reduced frame bundle of CP^n}).
	\end{itemize}
	
	\subsection*{Organization of the Work}
	The structure of the paper is as follows:
	\begin{itemize}
		\item \textbf{Section 2} presents the necessary geometric preliminaries on principal $G$-bundles, connections, and curvature two-forms. The basic knowledge about almost complex geometry is also discussed. 
		\item \textbf{Section 3} discusses the constructions of almost complex structures on principal bundles induced by choices of connections.
		\item \textbf{Section 4} introduces homogeneous bundles and classifies invariant connections, as well as invariant and integrable almost complex structures, in terms of linear maps between Lie algebras.
		\item \textbf{Section 5} develops the foundational theory of Hermitian and K\"ahler metrics on principal bundles. In the homogeneous setting, we provide the classification of K\"ahler metrics of a specific form.
		\item \textbf{Section 6} provides concrete geometric applications of the main theorems, specifically analysing the reduced frame bundles of the upper half-plane and complex projective spaces.
	\end{itemize}
	
	\subsection*{Notational Conventions}
	We fix the notational conventions adopted throughout the work. For a smooth manifold $M$, the space of sections of differential $k$-forms on a smooth manifold $M$ is denoted $\Omega^k(M)$. Given a vector space $V$, we denote by $\Omega^k(M;V)$ the space of sections of $V$-valued differential $k$-forms on $M$. If $E\to M$ is a vector bundle over $M$, the space of sections of $E$ is denoted $\Gamma(E)$ and the space of sections of $E$-valued differential $k$-forms is denoted by $\Omega^k(M;E)$.  The symbols $G$ and $H$ are reserved for Lie groups, with $H$ typically denoting a complex Lie group. The adjoint representation of a Lie group $G$ on its Lie algebra $\mfg$ is denoted $\Ad$, while the corresponding representation of $\mfg$ on itself is denoted by $\on{ad}$.
	
	\subsection*{Acknowledgements} The author is deeply grateful to his master thesis supervisor, Dr. Arash Bazdar, for his invaluable guidance, continuous support, and insightful mathematical discussions throughout the development of this work.
	
	\newpage
	
	\section{Preliminaries}
	In this section, we introduce the notions of princiapl $G$-bundles, connections on them, and tensorial forms. These are well-known and may be found in standard references such as \cite{loringtudifferentialgeometry} and \cite{kobayashi1}. All the manifolds and maps are assumed to be smooth unless explicitly stated otherwise.
	
	\subsection{Principal $G$-Bundles}
	Let $\pi:P\to M$ be a principal $G$-bundle. The vertical subbundle $VP$ of $T P$ is defined to be the kernel of $\pi:TP\to TM$.
	Given $p\in P$, it follows that the vertical subspace $V_p P$ consists of all fundamental vectors at $p$:
	\[V_p P=\{\xi_p^\#:\xi\in \mfg\},\]
	where $\xi_p^\#:=\displaystyle\frac{d}{dt}\Big|_{t=0} p\exp(t\xi).$
	
	A \emph{connection} on $P$ is a smooth choice $p\to A_p$ of subspaces $A_p$ of $T_p P$ such that for all $p\in P$, \[T_p P=V_p P\oplus A_p,\] 
	and for all $g\in G,p \in P$, $R_{g*}(A_p)=A_{pg}$, where $R_g: P\to P$ is the map sending $p$ to $pg$. We call $A_p$ the horizontal subspaces. Equivalently, a connection is a Lie-algebra valued one-form  $\theta:TP\to \mfg$ on $P$ such that
	\begin{enumerate}
		\item $\theta$ is $G$-equivariant. This means that for all $g\in G$, $R_g^*\theta=\Ad_{g^{-1}}\theta$.
		\item For all $p\in P$ and $\xi\in \mfg$ ,  $\theta_p(\xi_p^\#)=\xi$.
	\end{enumerate}
	The horizontal distribution $A$ is defined by $A:=\displaystyle\bigsqcup_{p\in P} A_p$. Elements of $A_p$ are called horizontal vectors at $p$ and sections of $A$ are called horizontal vector fields.
	Since a horizontal subbundle is complementary to the vertical subbundle, it follows that for any vector field $X$ on $M$, there exists a unique horizontal vector field $\tilde{X}\in \Gamma(A)$, such that $\pi_*(\tilde{X})=X$. We call $\tilde{X}$ the horizontal lift of $X$.
	
	The notion of \emph{tensorial forms} is central in our study.
	
	\begin{definition}[Tensorial forms]
		Let $P$ be a principal $G$-bundle over $M$ and $\rho:G\to \mathrm{GL}(V)$ be a representation of $G$ on a vector space $V$. A $V$-valued $k$-form is called a tensorial form of type $\rho$ if it is right-equivariant of type $\rho$ and it is horizontal. The first condition means that for all $g\in G$, 
		\[R_g^*\Phi=\rho(g^{-1})\Phi,\]
		while the second condition is satisfied if whenever one of $X_1,\dots, X_k$ is vertical, then 
		\[\Phi(X_1,\dots,X_k)=0.\]
		
		The space of all (smooth) tensorial $V$-valued $k$-forms of type $\rho$ is denoted by $\Omega_\rho^k(P,V)$.
	\end{definition}
	
	Tensorial forms have close relation with connections.
	
	\begin{proposition}[$\mca(P)$ as an Affine Space Modeled on $\Omega_{\Ad}^1(P,\mathfrak{g})$]~\\
		The space of connections on a principal bundle $P$, $\mca(P)$ is an affine space modeled on the space of all $\Ad$-tensorial one forms on $P$, $\Omega_{\mathrm{Ad}}^{1}(P,\mathfrak{g})$. More precisely, if $\theta$ is a connection on $P$ and $\Phi$ is a tensorial one-form, then $\theta+\Phi$ is a connection on $P$. Conversely, by fixing $\theta\in \mathcal{A}(P)$, every connection is of the form $\theta+\Phi$ for some tensorial one-form $\Phi$.
		\label{space of connections as affine space}
	\end{proposition}
	
	\begin{definition}[Covariant Derivative on Tensorial Forms]
		Let $P\to M$ be a principal $G$-bundle and $A$ a connection on $P$. Let $\rho:G\to \mathrm{GL}(V)$ be a representation of $G$ on a vector space $V$. The \emph{covariant derivative} $D^A\Phi$ of a tensorial $k$-form $\Phi\in \Omega_\rho^k(P;V)$ is defined to be
		\[D^A\Phi:=(\mrd\Phi)^h,\]
		where $\mrd\Phi$ is the exterior derivative on vector-valued forms, and 
		\[
		(\mrd\Phi)^h(X_1,\dots,X_{k+1}):=\mrd\Phi(hX_1,\dots,hX_{k+1}),
		\] where $X_i\in \Gamma(TP)$ and $hX_i$ is the horizontal component of $X_i$ with respect to the connection $A$.
	\end{definition}
	
	Finally, we introduce the curvature two-form of a connection.
	\begin{definition}[Curvature Two-Form]
		Let $A$ be a connection on a principal $G$-bundle $P$ over $M$ and $\theta: TP\to \mfg$ be the corresponding connection one-form. The curvature two-form of the connection $A$, denoted by $F_A$, is defined by
		\[F_A:=D^A\theta=(\mrd\theta)^h,\]
		where $(\mrd\theta)^h(X,Y):=\mrd\theta(hX,hY)$ for $X,Y\in \Gamma(TP)$.
	\end{definition}
	
	\begin{proposition}
		The curvature two-form $F_A$ of a connection $\theta$ on $P$ is a tensorial two-form of type $\Ad$. Moreover, it is given by the equation
		\[F_A=\mrd \theta+\frac{1}{2}[\theta,\theta],\]
		where $[\theta,\theta]$ is the $\mfg$-valued two-form on $P$ defined by
		\[[\theta,\theta](X,Y)=2[\theta(X),\theta(Y)], \quad X,Y\in \Gamma(TP).\]
		\label{curvature 2 form formula}
	\end{proposition}
	
	\vspace{-1cm}
	\subsection{Almost Complex Structures}
	We review the notion of almost complex structures. The materials presented are standard, with the primary references being \cite{johnleecomplexmanifolds} and Chapter IX of \cite{kobayashi2}.
	
	We begin by discussing almost complex structures on any manifold. We first introduce complex structures on vector spaces, and then define almost complex structures on a manifold.

	\subsubsection{Complex Structures on Vector Spaces}
	Let $V$ be a real vector space of dimension $2n$. 
	
	\begin{definition}[Complex Structures on Vector Spaces]
		A complex structure $J:V\to V$ on a real vector space $V$ is a linear map such that $J^2=-\operatorname{id}$.
	\end{definition}
	
	Every complex vector space $V$ has a canonical complex structure, which is the map sending $v\in V$ to $iv\in V$. Note that since $(\det J)^2=(-1)^{\displaystyle\dim_\mbr V}$, the dimension of $V$ must be even. It is called a complex structure because it helps to form a complex vector space as the next proposition shows.
	
	\begin{proposition}
		Let $V$ be a real vector space of dimension $2n$ and $J$ a complex structure on $V$. Define the scalar multiplication of $v\in V$ by a complex number by $(a+bi)v:=av+bJv$. With the original addition and the scalar multiplication by complex numbers defined, the real vector space $V$ becomes a complex vector space of dimension $n$.
	\end{proposition}
	
	Aside from turning $V$ into a complex vector space using a complex structure, another way is known as  the complexification of $V$. Consider the real vector space $V\otimes_\mbr \mbc$ which has dimension $4n$. We complexify $V\otimes_\mbr \mbc$ as follows. Given an element $v\otimes z\in V\otimes_\mbr \mbc$ and $\lambda\in \mbc$, we define $\lambda\cdot (v\otimes z)=v\otimes(\lambda z)$. With the original addition and the $\mbc$-scalar multiplication defined above, we denote the resulting complex vector space by $V_\mbc$, called the \emph{complexification of $V$}. It has complex dimension $2n$, that is 
	\[\dim_\mbr V=\dim_\mbc V_\mbc.\]
	Moreover, $V^\mbc$ is isomorphic to $V+iV:=\{v+iw:v,w\in V\}$ as complex vector spaces. For convenience, we consider the later vector space in the following discussion.
	
	Every linear map $T:V\to V$ between a real vector space can be uniquely extended to a complex linear map $T_\mbc:V_\mbc\to V_\mbc$ by defining $T_\mbc(v+iw):=Tv+iTw$. When no confusion arises, we shall denote this extension also by $T$. Now suppose $J:V\to V$ is a complex structure on a real vector space $V$. The complex linear extension $J:V_\mbc\to V_\mbc$ still satisfies $J^2=-\operatorname{id}$. Therefore, the eigenvalues of $J:V_\mbc\to V_\mbc$ are $\pm i$. We denote the eigenspace of $J$ with eigenvalue $i$ by $V^{1,0}$ and the eigenspace with eigenvalue $-i$ by $V^{0,1}$. The next proposition characterises $V^{1,0}$ and $V^{0,1}$.
	
	\begin{proposition}
		Let $J: V\to V$ be a complex structure on a real vector space $V$. Then 
		\[V^{1,0}=\{v-iJv:v\in V\},\]
		and
		\[V^{0,1}=\{v+iJv: v\in V\}.\]
		Moreover, the dimension of $V^{1,0}$ is equal to $V^{0,1}$.
		\label{V^1,0 and V^0,1}
	\end{proposition}
	
	Consequently, the vector space $V_\mbc$ decomposes as a direct sum of $V^{1,0}$ and $V^{0,1}$. Based on this decomposition, we define the projection maps.
	
	\begin{definition}
		Let $V$ be a vector space equipped with a complex structure $J$. Define the projection maps $\pi^{1,0}:V_\mbc\to V^{1,0}$ and $\pi^{0,1}:V_\mbc\to V^{0,1}$ by
		\[\pi^{1,0}(v):=\frac{1}{2}(v-iJv),\quad \pi^{0,1}(v):=\frac{1}{2}(v+iJV),\quad v\in V_\mbc.\]
		Note that every $v\in V_\mbc$ can be written as $v=\pi^{1,0}(v)+\pi^{0,1}(v)$.
	\end{definition}

	\subsubsection{Almost Complex Structures}
	In this section, we recall the definition of almost complex structures and study the structures they induce on the complexified tangent and the cotangent bundles of the manifold. 
	
	\begin{definition}[Almost Complex Structures]
		Let $M$ be a smooth manifold. An \emph{almost complex structure} $J:TM\to TM$ on $M$ is a (bundle) map such that $J^2=-\operatorname{id}$. In this case, we call $(M,J)$ an almost complex manifold.
	\end{definition}
	
	Since $J_x:T_xM\to T_xM$ is a complex structure on $T_x M$ for all $x\in M$, the preceding discussion about complex structures is applicable. We complexify each tangent space $T_x M$ and call the resulting collection of complexified tangent spaces the complexified tangent bundle of $M$, denoted by $T_\mbc M$. Note that it is now a smooth complex vector bundle (every fibre is a complex vector space). Using Proposition \ref{V^1,0 and V^0,1}, we get the following.
	
	\begin{proposition}
		Let $T^{1,0}M,T^{0,1}M$ be the vector subbundles of $T_\mbc M$ such that their fibres at $x\in M$ are $(T_xM)^{1,0}$ and $(T_xM)^{0,1}$ respectively. The following are true.
		\begin{enumerate}
			\item $T^{1,0}M=\{X-iJX:X\in \Gamma(TM)\}$ and $T^{0,1}M=\{X+iJX: X\in \Gamma(TM)\}$.
			\item The complexified tangent bundle $T_\mbc M$ can be decomposed into $T_\mbc M=T^{1,0}M\oplus T^{0,1}M$.
		\end{enumerate}
	\end{proposition}
	
	Denote by $\pi^{1,0}:T_\mbc M\to T^{1,0}M$ and $\pi^{0,1}:T_\mbc M\to T^{0,1}M$ the respective projection maps with respect to the decomposition $T_\mbc M=T^{1,0}M\oplus T^{0,1}M$. We can characterise $T^{1,0}M$ and $T^{0,1}M$ further if $M$ is a complex manifold. Before that, we define a complex manifold.
	
	\begin{definition}[Complex Manifolds]
		An $n$-dimensional complex manifold $M$ is a topological manifold equipped with a holomorphic atlas $\{(U_\alpha,\phi_\alpha:U_\alpha\to \mbc^n)\}_\alpha$, that is, an open covering $\{U_\alpha\}_\alpha$ of $M$ and local homeomorphisms $\{\phi_\alpha\}_\alpha$ such that if $U_{\alpha\beta}:=U_{\alpha}\cap U_{\beta}\neq \emptyset$, then $\phi_\alpha \circ \phi_{\beta}^{-1}:\phi_\beta(U_{\alpha\beta})\to \phi_\alpha(U_{\alpha\beta})$ is a biholomorphism. 
		
		Here, a biholomorphism is a bijective holomorphic map between open subsets of $\mbc^n$ such that the inverse is also holomorphic. A map $\phi_\alpha:U_\alpha \to \mbc^n$ in the holomorphic atlas is called a holomorphic chart.
	\end{definition}

	We proceed by introducing some common notations in complex geometry. Let $M$ be a complex manifold and $(z_1,\dots,z_n)$ a local holomorphic chart. By writing $z_j=x_j+iy_j$, we see that $M$ is a real manifold of dimension $2n$ with local coordinate chart $(x_1,\dots,x_n,y_1,\dots,y_n)$. For $1\leq j\leq n$, let
	
	\[\frac{\partial}{\partial z^j}:=\frac{1}{2}\left(\frac{\partial}{\partial x^j}-i\frac{\partial}{\partial y^j}\right),\quad\frac{\del}{\del \bar{z}^j}:=\frac{1}{2}\left(\frac{\del}{\del x^j}+i\frac{\del}{\del y^j}\right).\]

	\begin{definition}[Canonical Almost Complex Structure on a Complex Manifold]
		Let $M$ be a complex manifold and $(z_1,\dots,z_n)$ be a local holomorphic chart. The canonical almost complex structure $J:TM\to TM$ on $M$ is the map 
		\[J\left(\frac{\del}{\del x^j}\right)=\frac{\del}{\del y^j},\quad J\left(\frac{\del}{\del y^j}\right)=-\frac{\del}{\del x^j}, \]
		where $z^j=x^j+iy^j$ for $1\leq j\leq n$ and $x^j$ and $y^j$ are viewed as real coordinates of $M$.
		More precisely, if $\phi:U\subseteq M\to \phi(U)\subseteq \mbc^n$ is a local holomorphic chart, then the canonical almost complex structure $J$ is obtained from the following commutative diagram:
		\begin{center}
			\begin{tikzcd}
				TM|_U \arrow[r, "J", dashed] \arrow[d, "\phi_*"] & TM|_U \arrow[d, "\phi*"] \\
				\mathbb{R}^{2n} \arrow[r, "J_{\mathbb{C}^n}"]    & \mathbb{R}^{2n}         
			\end{tikzcd}
		\end{center} 
		where $J_{\mbc^n}$ is the canonical almost complex structure on $\mbr^{2n}$.
		\label{natural almost complex structures on complex manifolds}
	\end{definition}
	
	\begin{remark}~
		\begin{itemize}
			\item The definition of the canonical almost complex structure on a complex manifold is independent of the charts by the Cauchy-Riemann equations.
			
			\item Using the canonical almost complex structure on $M$, we see that
			\[\frac{\del}{\del z^j}=\pi^{1,0}\left(\frac{\del}{\del x^j}\right),\quad \frac{\del}{\del \bar{z}^j}=\pi^{0,1}\left(\frac{\del}{\del x^j}\right)\]
		\end{itemize}
	\end{remark}

	\begin{proposition}
		Let $M$ be a complex manifold and $(U,z^1,\dots,z^n)$ be a local holomorphic chart. Then with respect to the canonical almost complex structure $J$ on $M$, 
		\[\left\{\frac{\del}{\del z^j}\right\}_{j=1}^n \text { and } \left\{\frac{\del}{\del \bar{z}^j}\right\}_{j=1}^n\]
		are local frames for $T^{1,0}U$ and $T^{0,1}U$ respectively.
	\end{proposition}
	
	We now construct another important structure on differential forms induced by almost complex structures, which are called \emph{complex $(p,q)$-forms}.
	
	Suppose that $(M,J)$ is an almost complex manifold. The almost complex structure $J:TM\to TM$ induces an almost complex structure on $T^*M$. To be precise, define $J^*:T^*M\to T^*M$ by sending $\omega\in T^*M$ to $\omega\circ J\in T^*M$. We shall denote it also by $J$ if no confusion arises. Similarly, we complexify $T^*M$ and denote it by $T_\mbc^*M$, which will be called the complexified cotangent bundle of $M$.
	
	\begin{proposition}
		Let $(M,J)$ be an almost complex manifold.
		\begin{enumerate}
			\item With respect to $J^*$, the complexified cotangent bundle can be decomposed into
			\[T_\mbc^*M=T_{1,0}M\oplus T_{0,1}M,\]
			where $T_{1,0}M$ and $T_{0,1}M$ are the $i$- and $(-i)$-eigenspaces of $J^*$ respectively. 
			They are characterised by the following.
			\[T_{1,0}M=\{\omega\in \Gamma(T_\mbc^*M):\omega(Z)=0 \text{ for all $Z\in \Gamma(T^{0,1}M)$}\},\]
			and
			\[T_{0,1}M=\{\omega\in \Gamma(T_\mbc^*M):\omega(Z)=0 \text{ for all $Z\in \Gamma(T^{1,0}M)$}\}.\]
			\item Let $U$ be an open subset of $M$. If $\{\omega^1,\dots,\omega^n\}$ is a local frame for $T_{1,0}U$. Then $\{\bar{\omega}^1,\dots,\bar{\omega}^n\}$ is a local frame for $T_{0,1}U$.
		\end{enumerate}
		\label{T_{1,0} and T_{0,1}}
	\end{proposition}
	
	Using properties of the exterior power, for $k\in \mbz_{\geq 0}$, we see that the exterior power of the complexified cotangent bundle decomposes into 
	\[\Lambda^k(T_\mbc^*M)=\bigoplus_{p+q=k}\left(\Lambda^pT_{1,0}M\otimes\Lambda^qT_{0,1}M\right).\]
	We emphasise that $\displaystyle\Lambda^k(T_\mbc^*M)$ is a smooth complex vector bundle and the elements are complex differential $k$-forms.
	
	\begin{definition}[$(p,q)$-forms]
		Let $(M,J)$ be an almost complex manifold and $p,q$ non-negative integers. Denote by $\Omega^{p,q}(M,J)$ (or simply $\Omega^{p,q}(M)$) the space of sections of $\Lambda^p T_{1,0}M\otimes \Lambda^q T_{0,1}M$. The elements of $\Omega^{p,q}(M)$ are called $(p,q)$-forms. Moreover, for $k\geq 0$,
		\begin{equation}
			\Omega^k(T_\mbc^*M)=\bigoplus_{p+q=k}\Omega^{p,q}(M),
			\label{decomposition into p,q forms}
		\end{equation}
		where $\Omega^k(T_\mbc^*M)$ is the space of sections of smooth complex differential $k$-forms.
	\end{definition}
	
	We have the following characterisations of $(p,q)$-forms when $M$ is a complex manifold.
	
	\begin{proposition}
		Let $(M,J)$ be a complex manifold and $p,q$ be nonnegative integers. Let $(U,z^1,\dots,z^n)$ be a local holomorphic chart. For $1\leq j\leq n$, define 
		\[\mrd z^j:=\mrd x^j+i \,\mrd y^j,\, \mrd \bar{z}^j:=\mrd x^j - i\,\mrd y^j,\]
		where $z^j=x^j+iy^j$. Then a local frame for $\Lambda^{p,q}(T_\mbc^*M)$ over $U$ is 
		\[\{\mathrm{d}z^{i_1}\wedge\cdots\wedge \mathrm{d}z^{i_p}\wedge \mathrm{d}\bar{z}^{j_1}\wedge \cdots\wedge \mathrm{d}\bar{z}^{j_q}: 1\leq i_1<\cdots < i_p\leq n, 1\leq j_1<\cdots < j_p\leq n\}.\] 
	\end{proposition}

	The following proposition is important in decomposing a complex form into a sum of $(p,q)$-forms.
	\begin{proposition}
		Let $(M,J)$ be an almost complex manifold.
		\begin{enumerate}
			\item Let $\omega\in \Omega^k(T_\mbc^*M)$ be a $(p,q)$-form. Then $\omega$ vanishes if more than $p$ arguments are $(1,0)$ vector fields or more than $q$ arguments are $(0,1)$ vector fields.
			\item  A complex $k$-form $\omega\in \Omega^k(T_\mbc^*M)$ is a $(p,q)$-form for some nonnegative integers $p,q$ if and only if it vanishes whenever more than $p$ arguments are $(1,0)$ vector fields or more than $q$ arguments are $(0,1)$ vector fields.
		\end{enumerate}
		\label{characterisation of p,q forms}
	\end{proposition}
	
	The notion of complex vector-valued $(p,q)$- forms extends naturally. Let $V$ be a complex vector space. The space of $V$-valued complex $k$-forms, denoted $\Omega^k(M;V)$, can be identified with $\Omega^k(M;\mbc)\otimes V$. We may extend this to vector-valued $(p,q)$-forms.
	
	\begin{definition}[Vector-Valued $(p,q)$-Forms]
		Let $(M,J)$ be an almost complex manifold and $V$ a complex vector space. For $p+q=k$, the space of vector-valued $(p,q)$-forms, denoted by $\Omega^{p,q}(M,J;V)$ is the tensor product
		\[\Omega^{p,q}(M,J;V)=\Omega^{p,q}(M,J)\otimes_\mbc V.\]
	\end{definition}

	\subsubsection{Almost Holomorphic Maps}
	We end this section by considering maps between almost complex manifolds. In this category, we are concerned with almost holomorphic (or almost complex) maps.
	
	\begin{definition}[Almost Holomorphic Maps]
		Let $F:(M,J_M)\to (N,J_N)$ be a smooth map between two almost complex manifolds. It is said to be almost holomorphic if 
		\[F_*\circ J_M=J_N\circ F_*,\]
		equivalently, if the following diagram commutes:
		\begin{center}
			\begin{tikzcd}
				TM \arrow[d, "F_*"'] \arrow[r, "J_M"] & TM \arrow[d, "F_*"] \\
				TN \arrow[r, "J_N"]                   & TN                 
			\end{tikzcd}
		\end{center}
		It is called an almost biholomorphism if it is bijective and its inverse is also almost holomorphic.
	\end{definition}
	
	The notion of almost holomorphicity generalises that of holomorphicity of maps between complex manifolds.
	\begin{proposition}
		Let $M$ and $N$ be complex manifolds equipped with canonical almost complex structures. A smooth map $F:M\to N$ is holomorphic if and only if it is almost holomorphic with respect to the canonical almost complex structures.
	\end{proposition}

	The next proposition says that an almost holomorphic map preserves structures induced by the respective almost complex structures.
	\begin{proposition}
		Let $(M,J_M)$ and $(N,J_N)$ be almost complex manifolds. Let $F:M\to N$ be an almost holomorphic map. The following are true.
		\begin{enumerate}
			\item $F_*$ maps $\Gamma(T^{1,0}M)$ to $\Gamma(T^{1,0}N)$ and $\Gamma(T^{0,1}M)$ to $\Gamma(T^{0,1}N)$.
			\item $F^*(\Omega^{p,q}N)\subseteq \Omega^{p,q}M$.
		\end{enumerate}
		\label{almost holomorphic maps preserve structures}
	\end{proposition}


	\subsubsection{Integrability of Almost Complex Structures}
	In this section, we state the important theorem - the Newlander-Nirenberg theorem, which connects almost complex structures and holomorphic structures of a manifold. The main object connecting them is the Nijenhius tensor.
	
	\begin{definition}[Nijenhius Tensor]
		Let $(M,J)$ be an almost complex manifold. The Nijenhius tensor of $J$ is defined by 
		\[N_J(X,Y)=[X,Y]+J[X,JY]+J[JX,Y]-[JX,JY].\]
	\end{definition}
	
	\begin{definition}[Integrability]
		Let $(M,J)$ be an almost complex structure. The almost complex structure $J$ is called integrable if $N_J\equiv 0$.
	\end{definition}
	
	\begin{theorem}[Newlander-Nirenberg Theorem]
		An almost complex manifold $(M,J)$ is a complex manifold such that $J$ becomes its canonical almost complex structure if and only if it is integrable, i.e., $N_J\equiv 0$.
		\label{newlander-nirenberg theorem}
	\end{theorem}
	
	We will use this to formulate an integrability condition of an almost complex structure on a principal bundle.

	\section{Almost Complex Structures on Principal Bundles}
	
	\subsection{Complexions on Principal Bundles}
	We now discuss the most important objects in this work, which are almost complex structures on a principal bundle. We will demonstrate the construction of such structures given that the base manifold is almost complex and the structure group is a complex Lie group. To emphasise that the Lie group is now complex, we use $H$ to represent the structure group and $\mfh$ to indicate its Lie algebra.
	The construction of such almost complex structures has been well studied. For example, see \cite{biswas}.
	
	For the remaining discussion in this section, let $(M,j)$ be an almost complex manifold and $\pi:P\to M$ a principal $H$-bundle over $M$ where $H$ is a complex Lie group. Denote by $i_H$ the canonical almost complex structure on $H$.
	
	\begin{definition}[Complexions]
		An almost complex structure $J$ on $P$ is called a \emph{complexion} (see \cite{biswas}) if the following hold:
		\begin{enumerate}
			\item The map $\pi:(P,J)\to(M,j)$ is almost holomorphic, namely 
			\[\pi_*\circ J = j\circ \pi_*.\]
			\item The action $P\times H\to P$ is almost holomorphic with respect to the respective almost complex structures.
		\end{enumerate}
	\end{definition}
	
	\begin{remark}
		In the second condition, the almost complex structure of $P\times H$ is taken to be $J\times i_H$. Let $h\in H$ and $p\in P$. Let $R_h:P\to P$ be the map sending $q\in P$ to $qh$ and $\ell_p:H\to P$ be the map sending $h'$ to $ph'$. The assumption implies that $R_h:(P,J)\to (P,J)$ and $\ell_p:(H,i_H)\to (P,J)$ are almost holomorphic. 
	\end{remark}
	
	\begin{proposition}
		Let $J$ be a complexion on a principal $H$-bundle over $(M,j)$. Then for every $h\in H$, $R_{h*}\circ J=J\circ R_{h*}$. Moreover, $J$ acts on the vertical subbundle $VP$ as follows: for any $\eta\in \mfh$ and $p\in P$, 
		\[J(\eta_p^\#)=(i_H\eta)_p^\#.\]
	\end{proposition}
	
	\begin{proof}
		The first statement follows from the remark above. To prove the second statement, since $\ell_{p,*}\circ i_H=J\circ \ell_{p,*}$, if $\eta\in \mfh$, we see that
		\[(i_H\eta)_p^\#=\frac{d}{dt}\Bigg|_{t=0}p\exp(i_H\eta)=\ell_{p,*}(i_H\eta)=J(\ell_{p*}\eta)=J(\eta_p^\#).\]
	\end{proof}
	
	Analogous to the fact that the space of connections is an affine space modeled on the space of tensorial one-forms of type $\Ad$, the space of complexions can likewise be realised as an affine space modeled on a suitable subspace of tensorial one-forms. To make this precise, we introduce tensorial $(p,q)$-forms of type $\Ad$.
	
	\begin{definition}[Tensorial $(p,q)$-forms]
		Let $J$ be a complexion on $P$. The space of tensorial $(p,q)$-forms of type $\Ad$, denoted by 
		\[\Omega_{\Ad}^{p,q}(P,J;\mfh),\]
		is the subspace of $\Omega^{p,q}(P,J;\mfh)=\Omega^{p,q}(P,J)\otimes \mfh$ consisting of $\mfh$-valued $(p,q)$-forms that is horizontal and $R_h^*\Psi=\Ad_{h^{-1}}\Psi$ for every $h\in H$.
	\end{definition}
	It is worth noting that a tensorial $(p,q)$-form is a complex differential form, and the Lie algebra $\mfh$ is regarded as a complex Lie algebra equipped with the complex structure $i_H$.
	
	\begin{theorem}[\cite{biswas}]
		Let $J^0$ be a fixed complexion on $P$. Then there is a one-to-one correspondence between the space of all complexions on $P$ and the space of all $(0,1)$-tensorial forms (relative to $J^0$) on $P$, $\Omega_{\Ad}^{0,1}(P,J^0;\mathfrak{h})$. More precisely, let $J$ be a complexion on $P$. Then there exists $\Psi\in \Omega_{\Ad}^{0,1}(P,J^0;\mathfrak{h})$ such that
		\begin{equation}
			J=J^0+(\Psi(\cdot))^\#.
			\label{J-J^0}
		\end{equation}
		On the other hand, for every $\Psi\in \Omega_{\Ad}^{0,1}(P,J^0;\mathfrak{h})$, the endomorphism $J$ defined by \eqref{J-J^0} is a complexion on $P$. Therefore, the space of complexions on $P$ is an affine space modeled on the space of $(0,1)$-tensorial forms.
		\label{space of complexions as an affine space}
	\end{theorem}
	
	\begin{proof}
		First note that $J-J^0$ is a map from $TP$ to $VP$, since $\pi_*(J-J^0)=\pi_*J-\pi_*J^0=j\pi_*-j\pi_*=0$. Therefore, for all $X\in TP$, there exists a unique $\Psi(X)\in\mathfrak{h}$ such that $(J-J^0)(X)=(\Psi(X))^\#$. This defines a linear map $\Psi:TP\to \mathfrak{h}$. Now we show $\Psi\in \Omega_{\Ad}^{(0,1)}(P,J^0;\mathfrak{h})$. If $X=\eta^\#$ is vertical, since $(J-J^0)(\eta^\#)=(i\eta)^\#-(i\eta)^\#=0,$ it follows that $\Psi(X)=0$. The Ad-equivariance of $\Psi$ follows from the computation
		\begin{align*}
			(R_h^*\Psi(X))_{ph}^\# &=(\Psi(R_{h*}X))_{ph}^\#\\
			&=(J_{ph}-J_{ph}^0)(R_{h*}(X_p))\\
			&=R_{h*}(J_p-J_p^0)(X_p)\\
			&=R_{h*}(\Psi_p(X_p)^\#)\\
			&=(\Ad_{h^{-1}}\Psi(X))_{ph}^\#.
		\end{align*}
		
		We continue to show that $\Psi$ is a $(0,1)$-form with respect to $J^0$. Since $J-J^0$ is a map from $TP$ to $VP$ and $J-J^0\mid_{VP}=0$, we have $(J-J^0)^2=0$. This can be simplified to 
		\begin{equation}
			JJ^0+2I=-J^0J.
			\label{JJ^0+J^0J}
		\end{equation}
		Let $X\in TP$. We see that
		\begin{align*}
			(i\Psi(J^0X))^\#&=J^0(\Psi(J^0X)^\#)\\
			&=J^0(J-J^0)J^0X\\
			&=(J^0J)J^0X+J^0X\\
			&=(-JJ^0-2I)J^0X+J^0X & (\text{by } \eqref{JJ^0+J^0J})\\
			&=JX-2J^0 X+J^0X\\
			&=(J-J^0)(X)\\
			&=(\Psi(X))^\#.
		\end{align*}
		Therefore, $\Psi$ is a $(0,1)$-form relative to $J^0$. The converse is immediate from computations.
	\end{proof}
	
	\subsection{Connections and Induced Complexions}
	In the previous section, we saw how a complexion $J$ acts on the vertical subbundle $VP$: If $p\in P,\eta\in \mfh$, then $J(\eta_p^\#)=(i_H\eta)_p^\#$. If a connection on $P$ is chosen, the almost holomorphicity of $\pi:P\to M$ allows us to construct a complexion. 
	
	\begin{definition}[Induced Complexions]
		Let $A$ be a connection on $P$. The complexion $J_A$ induced by $A$ is defined by
		\begin{itemize}
			\item On $VP$, $J_A(\eta^\#):=(i_H\eta)^\#$, $\eta\in \mfh$.
			\item On $A$, $J_A:=(\pi_*\mid_{A})^{-1}\circ j\circ \pi_*$. 
		\end{itemize}
		Hence, if $\tilde{X}\in \Gamma(A)$ is the horizontal lift of a vector field $X$ on $M$, then $J_A(\tilde{X})=\widetilde{j(X)}$.
	\end{definition}
	
	As an immediate consequence, we have the following.
	\begin{proposition}
		Let $A$ be a connection on $P$. The induced complexion $J_A$ respects the Whitney sum decomposition $TP=VP\oplus A$:
		\[J(VP)=VP,\quad J(A)=A.\]
	\end{proposition}
	
	The following theorem gives the relation between the corresponding tensorial forms of a connection and its induced complexion.
	
	\begin{theorem}[\cite{biswas}]~
		\begin{enumerate}
			\item If two connections $\theta$ and $\theta^0$ are related by $\Phi\in \Omega_{\Ad}^1(P,\mathfrak{h})$: $\theta=\theta^0+\Phi$, then their respective induced complexions $J$ and $J^0$ are related by $\Psi$, where \[\Psi=i\Phi-\Phi J^0.\]
			\item Two connections $\theta$ and $\theta^0$ that are related by $\Phi$ induce the same complexions if and only if $\Phi$ is a $(1,0)$-form with respect to $J^0$.
			\item If $\theta^0$ induces $J^0$ and $J$ is related to $J^0$ by $\Psi$, then $J$ is induced by the connection \[\theta=\theta^0-\frac{i}{2}\Psi.\]
		\end{enumerate}
		\label{relation between tensorial forms of connection and complexion}
	\end{theorem}
	
	\begin{proof}~
		\begin{enumerate}
			\item Let $A^0$ and $A$ be the horizontal distributions of $\theta^0$ and $\theta$ respectively. Let $X\in \Gamma(A^0)$. The horizontal component of $X$ with respect to $A$ is \[X^h=X-(\theta(X))^\#=X-(\theta^0(X)-\Phi(X))^\#=X-(\Phi(X))^\#.\]
			Now,
			\begin{align*}
				(\Psi(X))^\#&=(J-J^0)(X)\\
				&=JX^h+J(\Phi(X)^\#)-J^0X\\
				&=JX^h+(i\Phi(X))^\#-J^0X.
			\end{align*}
			Next, we observe that $\pi_*(J(X^h))=j(\pi_*(X^h))=j(\pi_*X)$ and
			\[\pi_*((J^0X)^h)=\pi_*(J^0X)=j(\pi_*X).\]
			Since $\pi_*:A\to TM$ is an isomorphism, we get $J(X^h)=(J^0X)^h$. The horizontal component of $J^0X$ with respect to $A$ is $J^0X-(\theta(J^0X))^\#=J^0X-(\Phi(J^0X))^\#$. Therefore, 
			\[(\Psi(X))^\#=(i\Phi(X))^\#-(\Phi J^0X)^\#.\]
			This proves the desired result.
			\item From (1), the connections $\theta$ and $\theta^0$ induce the same complexion if and only if $\Psi=0$. This is equivalent to $\Phi J^0=i\Phi$.
			\item Let $\theta=\theta^0-\frac{i}{2}\Psi$. By (1), the complexion $J^\theta$ induced by $\theta$ is 
			\[J^\theta=J^0+\left(i\left(-\frac{i}{2}\Psi\right)+\frac{i}{2}\Psi J^0\right)=J^0+(\Psi(\cdot))^\#=J.\]
			Therefore, $J$ is induced by $\theta$.
		\end{enumerate}
	\end{proof}
	
	It turns out that all complexions are induced complexions (we do not require a complexion to be induced by a connection in the definition).
	\begin{corollary}
		Every complexion on a principal $H$-bundle is induced by a connection.
	\end{corollary}
	\begin{proof}
		Let $\theta^0$ be a connection on $P$, and $J^0$ be its induced complexion. Suppose $J$ is a complexion on $P$. Then $J$ is related to $J^0$ by some $\Psi\in \Omega_{\Ad}^{0,1}(P,J^0;\mathfrak{h})$ by Theorem \ref{J-J^0}. By (3) of Theorem \ref{relation between tensorial forms of connection and complexion}, the complexion $J$ is induced by the connection $\theta=\theta^0-\frac{i}{2}\Psi$.
	\end{proof}

	\subsubsection{Integrability and Curvature Forms}
	In this section, we formulate a condition for the integrability of a complexion. It is related to the curvature two-form of the connection that induces the complexion. We begin with the following.

	\begin{lemma}
		Let $X\in \Gamma(TM)$ and $\eta\in \mfh$. Then $[\tilde{X},\eta^\#]=0$.
	\end{lemma}
	\begin{proof}
		Recall that $\tilde{X}$ is $H$-invariant and the one-parameter subgroup of $\eta^\#$ is $R_{\exp(t\eta)}$. Hence,
		\[\begin{aligned}
			[\tilde{X},\eta^\#]
			&=-\frac{d}{dt}\Big|_{t=0}R_{-\exp(t\eta)*}(\tilde{X})\\
			&=-\frac{d}{dt}\Big|_{t=0}\tilde{X}\\
			&=0.
		\end{aligned}\]
	\end{proof}
	
	\begin{proposition}
		Let $(P,J)$ be a principal $H$-bundle over $(M,j)$ where $J$ is a complexion induced by a connection $A$ with curvature form $F_A$. Let $N_J$ and $N_j$ be the Nijenhius tensors of $J$ and $j$ respectively.
		Let $U,V\in \Gamma(TP)$, then
		\[N_J(U,V)=N_j(\pi_*U,\pi_*V)-(F_A^{0,2}(U,V))^\#,\]
		where $F_A^{0,2}$ is the $(0,2)$-part of the curvature $2$-form $F_A$ of $A$, given by
		\[F_A^{0,2}(\cdot,\cdot)=F_A(\cdot,\cdot)-F_A(J\cdot,J\cdot)+iF_A(J\cdot,\cdot)+iF_A(\cdot,J\cdot).\]
		\label{Nijenhius tensors of base and total space}
	\end{proposition}
	
	\vspace{-0.5cm}
	
	\begin{proof}
		If either $U$ or $V$ is vertical, then $N_J(U,V)=0$ by the previous lemma and the fact that $i_H$ commutes with the adjoint action of $H$. The right hand side of the equation also vanishes. Now suppose $U,V$ are horizontal lifts. Let $X,Y\in \Gamma(TM)$. From $[\tilde{X},\tilde{Y}]=\widetilde{[X,Y]}-(F_A(\tilde{X},\tilde{Y}))^\#$, we compute that
		\begin{align*}
			N_J(\tilde{X},\tilde{Y})
			&=[\tilde{X},\tilde{Y}]+J[J\tilde{X},\tilde{Y}]+J[\tilde{X},J\tilde{Y}]-[J\tilde{X},J\tilde{Y}]\\
			&=\widetilde{[X,Y]}+J\widetilde{[j(X),Y]}+J\widetilde{[X,j(Y)]}-\widetilde{[j(X),j(Y)]}\\
			&\ \ \ \ \  - (F_A(\tilde{X},\tilde{Y}))^\#-J(F_A(J\tilde{X},\tilde{Y}))^\#-J(F_A(\tilde{X},J\tilde{Y}))^\#+(F_A(J\tilde{X},J\tilde{Y}))^\#\\
			&=\widetilde{N_j(X,Y)}-(F_A(\tilde{X},\tilde{Y})+iF_A(J\tilde{X},Y)+iF_A(\tilde{X},J\tilde{Y})-F_A(J\tilde{X},J\tilde{Y}))^\#\\
			&=\widetilde{N_j(X,Y)}-(F_A^{0,2}(\tilde{X},\tilde{Y}))^\#.
		\end{align*}
		Here we used the facts that $J\eta^\#=(i\eta)^\#$ and $J\tilde{X}=\widetilde{j(X)}$ for $\eta\in \mfh,X\in \Gamma(TM)$.
	\end{proof}
	
	As a direct consequence, we have the important corollary below.
	\begin{corollary}[\cite{biswas}]
		Let $(M,j)$ be an almost complex manifold with an integrable almost complex structure ($N_j=0$). Let $P$ be a principal $H$-bundle equipped with a connection, and $J$ be the induced complexion. Then $(P,J)$ is integrable if and only if the $(0,2)$-part of the curvature form vanishes. 
		\label{integrability condition and F_A}
	\end{corollary}
	\begin{proof}
		By Proposition \ref{Nijenhius tensors of base and total space} and since $N_j\equiv 0$, we see that
		\[N_J=(F_A^{0,2})^\#.\]
		Hence, $J$ is integrable if and only if $F_A^{0,2}\equiv 0$.
	\end{proof}
	
	In the next section, we will use this condition to determine the integrability of almost complex structures on homogeneous bundles.
	
	\newpage
	
	\section{Equivariant Homogeneous Bundles}
	In this section, we study equivariant homogeneous bundles, namely principal $H$-bundles over homogeneous spaces $M=G/K$, where the action of $G$ on $M=G/K$ admits a lift to the total space of the bundle. Note that we use the symbol $G$ for the Lie group acting on the base, and the symbol $H$ for the structure group of the principal bundle. Section 4.1 introduces homogeneous spaces and presents their classification in terms of Lie groups. In Section 4.2, we define homogeneous bundles and provide some examples. Sections 4.3 and 4.4 are devoted to the study of invariant connections on homogeneous bundles. In particular, two correspondence theorems of invariant connections (Theorem \ref{Wang classification theorem} and Theorem \ref{classification of invariant connections (Gx_lambda H)}) will be established. In Section 4.5, a correspondence theorem of invariant complexions will be given. 
	
	\subsection{Homogeneous Spaces}
	In this section, we recall the definition of homogeneous spaces and explore their properties without proofs. For the following discussion, we denote the Lie group of a homogeneous space by $G$.
	\begin{definition}[Homogeneous Spaces]
		Let $G$ act smoothly on a manifold $M$ (on the left), that is, the map $G\times M\to M$ sending $(g,x)$ to $g\cdot x$ is smooth. The manifold $M$ is called a \emph{homogeneous space} for $G$ if $G$ acts transitively on $M$. We may sometimes write a homogeneous space $M$ for a Lie group $G$ as $(M,G)$.
	\end{definition}
	
	Throughout the discussion of homogeneous spaces and homogeneous bundles (which will be introduced in the next section), we choose a basepoint $x_0\in M$ and consider the isotropy subgroup $K$ of $G$ at $x_0$, 
	\[K=\mathrm{Stab}_{G}(x_0):=\{g\in G: g\cdot x_0 = x_0\}.\]
	It is a closed subgroup of $G$, hence a Lie subgroup of $G$. Let $\mfg$ and $\mfk$ be the Lie algebras of $G$ and $K$ respectively. By the orbit-stabiliser theorem for a group action, $G/K$ is bijective to the orbit space of $x_0$ under the transitive action of $G$ on $M$. Hence $G/K$ is bijective to $M$. In fact, every homogeneous space arises as a quotient $G/K$, so the classification of them reduces to the study of Lie groups.
	
	\begin{proposition}
		Let $(M,G)$ be a homogeneous space. Then $M$ is diffeomorphic to the coset manifold $G/K$. Conversely, every coset manifold $G/K$, where $K$ is a closed subgroup of $G$, is a homogeneous space where the action is given by sending $(g,g'K)\in G\times (G/K)$ to $gg'K\in G/K$. Moreover, the map $\pi:G\to G/K$ which sends $g\in G$ to its coset $gK\in G/K$ is a submersion.
		\label{homogeneous spaces classification}
	\end{proposition}
	
	Therefore, we may, when convenient, identify $M$ with $G/K$. Denote by $r_{x_0}:G\to M$ the map that sends $g\in G$ to $gx_0\in M$.
	
	\begin{proposition}
		Let $(M,G)$ be a homogeneous space.
		\begin{enumerate}
			\item The kernel of $(r_{x_0})_{*,e}:\mfg\to T_{x_0}M$ is $\mfk$, the Lie algebra of $K$.
			\item The map $r_{x_0*}:\mfg\to T_{x_0} M$ restricts to an isomorphism $(r_{x_0})_{*,e}:\mfg/\mfk \to T_{x_0}M$. In particular,
			\[T_{x_0}(G/K)\cong\mfg/\mfk.\]
		\end{enumerate}
	\end{proposition}

	\subsection{Equivariant Homogeneous Bundles}
	We begin by defining equivariant and equivariant homogeneous bundles.
	\begin{definition}[Equivariant Bundles]
		Let $\pi:P\to M$ be a principal $H$-bundle over $M$ ($H$ is not necessarily complex). Suppose $G$ acts on $P$ on the left. The bundle $P$ is called an equivariant bundle if
		\begin{enumerate}
			\item $\pi(gp)=g\pi(p)$,
			\item $g(ph)=(gp)h.$
		\end{enumerate}
		If both conditions are satisfied, we say that the action of $G$ on $M$ can be lifted to an action of $G$ on $P$.
	\end{definition}
	The first condition says that $p\in P$ upon action of an element $g\in G$ is projected to the same basepoint of $p$ but translated by $g$. In particular, if $k\in K$ and $\pi(p)=x_0$, then $kp\in P_{x_0}$. Therefore, the Lie group $K$ acts on $P_{x_0}$.
	
	\begin{definition}[Equivariant Homogeneous Bundles]
		An equivariant homogeneous bundle (or simply homogeneous bundle) is an equivariant bundle over a homogeneous space.
	\end{definition}

	\begin{example}[Frame Bundles]
		Let $M$ be a homogeneous space for a Lie group $G$. Let $E\to M$ be a vector bundle and $\mathrm{Fr}(E)$ its frame bundle. Denote by $L_g:M\to M$ the map sending $x\in M$ to $g\cdot x\in M$. The action of $G$ on $M$ induces an action of $G$  on $\Gamma(E)$, by mapping $X$ to $L_{g*}(X)$. Therefore, it induces an action of $G$ on $\mathrm{Fr}(E)$ by
		\begin{align*}
			G\times \mathrm{Fr(E)} &\to \mathrm{Fr}(E):\\
			(g,(x,[v_1,\cdots,v_r]))&\mapsto (g\cdot x,[(L_{g*})v_1,\cdots (L_{g*})v_r]).
		\end{align*}
		The frame bundle $\mathrm{Fr}(E)$ is then a homogeneous bundle over $M$. For instance, 
		\begin{align*}
			\pi(g\cdot(x,[v_1,\cdots,v_r]))&=\pi(g\cdot x,[(L_{g*})v_1,\cdots (L_{g*})v_r])\\
			&=g\cdot x\\
			&=g\cdot \pi(x,[v_1,\cdots,v_r]).
		\end{align*}
	\end{example}
	
	\begin{example}[Homogeneous Spaces]
		Let $M$ be a homogeneous space for a Lie group $G$. It follows that $G\to M$ is a principal $G$-bundle. We have a natural action of $G$ on $G$ by left multiplication. Consequently, $G\to M$ is a homogeneous bundle.
	\end{example}
	
	Similar to the classification of homogeneous spaces as coset manifolds, our goal now is to classify homogeneous bundles in terms of their Lie groups. For this purpose, we introduce the following bundle.
	
	\begin{proposition}
		Let $G,H$ be two Lie groups and $M=G/K$ a homogeneous space, where $K$ is a closed subgroup of $G$. Let $\lambda:K\to H$ be a (Lie group) homomorphism. Define the set $G\times_\lambda H$ by 
		\[G\times_\lambda H:=\{[g,h]_\lambda=[gk^{-1},\lambda(k)h]_\lambda: g\in G,h \in H,k \in K\}.\]
		Define the action of $G$ on $G\times_\lambda H$ by 
		\[g'\cdot[g,h]_\lambda:=[g'g,h]_\lambda, \quad g,g'\in G, h\in H,\]
		and the action of $H$ by 
		\[[g,h]_\lambda\cdot h':=[g,hh'],\quad g\in G, h,h'\in H.\]
		Define the projection map $\pi_\lambda:G\times_\lambda H\to M$ by
		\[\pi_\lambda([g,h]_\lambda):=gK.\]
		Then $G\times_\lambda H$ has a manifold structure such that $\pi_\lambda:G\times_\lambda H\to M$ is a principal $H$-bundle over the homogeneous space $M$. Moreover, $G\times_\lambda H$ is a homogeneous bundle.
	\end{proposition}
	
	\begin{proof}
		We omit the proof of existence of the manifold structure. By Proposition \ref{local trivialisation and local section}, it suffices to find a local section of $G\times_\lambda H$. Since $G\to G/K$ is a principal $K$-bundle, we have a collection of local sections of $G\to G/K$. Let $\sigma:U \to G$ be a local section of $G\to G/K$. Define the local section $s^\sigma:U\to G\times_\lambda H$ of $G\times_\lambda H$ by $s^\sigma(x):=[\sigma(x),e]_{\lambda}$ where $x\in U$ and $e$ is the identity element of $H$. 
		
		To check that $G\times_\lambda H$ is a homogeneous bundle, note that $\pi_{\lambda}(g'\cdot[g,h]_\lambda)=g'gK=g'\pi_\lambda([g,h]_\lambda)$ for $g,g'\in G$ and $h\in H$.
	\end{proof}
	
	\begin{proposition}
		Let $\lambda:K\to H$ and $\lambda':K\to H$ be two homomorphisms. Then $G\times_\lambda H$ is isomorphic to $G\times_{\lambda'}H$ if and only if $\lambda'=c_h\circ \lambda$ for some $h\in H$, where $c_h:H\to H$ is the conjugation map by $h$.
		\label{isomorphism class of G times H by conjugation}
	\end{proposition}
	
	\begin{proof}
		Let $P=G\times_\lambda H,P'=G\times_{\lambda'}H$ and $\lambda'=c_h\circ\lambda$ for some $h\in H$. Define the map $F^h:P\to P'$ by
		\[F^h[g,h']_{\lambda}:=[g,hh']_{\lambda'}.\]
		This is well defined since $[gk^{-1},h(\lambda(k)h')]_{\lambda'}=[gk^{-1},\lambda'(k)hh']_{\lambda'}=[g,hh']_{\lambda'}$ for all $k\in K$. One may check that this is an $\operatorname{id}_M$-covering bundle map and compatible with the group action. Clearly it is bijective. Therefore, $P$ and $P'$ are isomorphic.
		
		Conversely, suppose that $F:P\to P'$ is a bundle isomorphism. Since $F$ is fibre-preserving and compatible with the group action, i.e., for all $g,h'\in H$, $F([g,h']_\lambda h)=F([g,h']_\lambda)h$, we see that $F([g,h']_\lambda)=[g,hh']_{\lambda'}$ for some $h\in H$. Moreover, since $F$ is well-defined, we get 
		\begin{align*}
			[g,hh']_{\lambda'}
			&=F([g,h']_\lambda)\\
			&=F([gk^{-1},\lambda(k)h']_\lambda)\\
			&=[g,\lambda'(k)^{-1}h\lambda(k)h']_{\lambda'}.
		\end{align*}
		for all $k\in K$. This shows that $\lambda'=c_h\circ \lambda$.
	\end{proof}

	\begin{theorem}[Isomorphism Classes of Homogeneous $H$-bundles]
		Let $\pi:P\to M$ be a homogeneous $H$-bundle over $M$. Fix $p_0\in \pi^{-1}(x_0)$.
		\begin{enumerate}
			\item There exists a homomorphism $\lambda:K\to H$ such that
			\[p_0\lambda(k)=kp_0,\, \text{ for all } k \in K.\]
			\item We have a bundle isomorphism 
			\begin{align*}
				\Phi:\,&G\times_\lambda H \to P\\
				&[g,h]_\lambda \hspace{0.3cm}\mapsto gp_0h.
			\end{align*}
			Consequently, every homogeneous $H$-bundle is isomorphic to a bundle of the form $G\times_\lambda H$ for some homomorphism $\lambda:K \to H$.
		\end{enumerate}
		\label{classification of homogeneous bundles}
	\end{theorem}
	
	\begin{proof}~
		\begin{enumerate}
			\item Let $k\in K$. Recall that $K$ acts on $P_{x_0}$. Hence, there exists a unique $h_k\in H$ such that $kp_0=p_0h_k$. We define $\lambda: K\to H$ by sending $k$ to $h_k$. To show $\lambda$ is a homomorphism, let $k_1,k_2\in K$. By the definition of $\lambda$, \[p_0\lambda(k_1k_2)=k_1k_2p_0=k_1(p_0\lambda(k_2))=p_0\lambda(k_1)\lambda(k_2).\]
			Since the action of $H$ is free, $\lambda(k_1k_2)=\lambda(k_1)\lambda(k_2).$
			\item We first show $\Phi$ is well-defined. This follows from the computation 
			\begin{align*}
				\Phi([gk,\lambda(k)^{-1}h]_\lambda)&=g(kp_0)\lambda(k)^{-1}h\\
				&=g(p_0\lambda(k))\lambda(k)^{-1}h\\
				&=gp_0h.
			\end{align*}
			
			Now we find the inverse of $\Phi$. Let $p\in P$. Since the $G$-action on $M$ is transitive, there exists $g\in G$ such that $\pi(p)=gx_0$. It follows that $g^{-1}p\in P_{x_0}$. Thus, there exists $h\in H$ such that $g^{-1}p=p_0 h$.
			The inverse of $\Phi$ is then given by $\Phi^{-1}(p)=[g,h]_\lambda$.
			Finally, to show $\Phi$ is a bundle morphism, note that \[\pi(\Phi([g,h]_\lambda))=\pi(gp_0h)=g\pi(p_0)=\pi_\lambda([g,h]_\lambda),\]
			where $\pi_\lambda:G\times_\lambda H\to M$ is the projection map for the bundle $G\times_\lambda H$.
		\end{enumerate}
	\end{proof}
	
	Therefore, in order to study homogeneous bundles, it suffices to study those of the form $G\times_\lambda H$ for some homomorphism $\lambda:K\to H$. Moreover, we have a complete classification of homogeneous bundles.
	
	\begin{definition}
		Let $\mathcal{M}=\mathrm{Hom}(K,H)/\sim_H$, where two homomorphisms are related if and only if they are a conjugation of each other by an element of $H$, and $\mathrm{Hom}(K,H)$ is the set of all (Lie group) homomorphisms from $K$ to $H$.
	\end{definition}
	
	\begin{corollary}
		There is a one-to-one correspondence between the set of isomorphism classes of homogeneous $H$-bundles over $M=G/K$ and the set of equivalence classes $\mathcal{M}$.
	\end{corollary}

	A central focus of the following discussion is on invariant connections and complexions. Therefore, we end this section by giving the definition of invariance. The group actions $G$ on $P$ and on $M$ naturally induce actions on their respective vector-valued differential forms by pullbacks. For instance, let $V$ be a finite-dimensional vector space. The action $G\times P\to P$ induces an action $G\times \Omega^k(P;V)\to \Omega^k(P;V)$ by
	\[(g,\omega)\mapsto L_g^*\omega,\]
	where $L_g:P\to P$ is the left multiplication by the element $g$.

	\begin{definition}[Invariant Forms]
		Let $V$ be a finite-dimensional vector space. Let $\omega\in \Omega^k(P;V)$. It is called a $G$-invariant (or simply invariant) form if $L_g^*\omega=\omega$ for all $g\in G$.
	\end{definition}
	
	\begin{remark}
		A connection one-form on a principal $H$-bundle $P$ is a $\mfh$-valued one-form. Hence, an invariant connection is a connection one-form that is invariant. Equivalently, if $A$ is the corresponding horizontal distribution, the invariance of the connection form means that $L_{g*}A=A$ for all $g\in G$. 
	\end{remark}
	
	The group action $G$ on $P$ and on $M$ also naturally induces an action on their respective endomorphism bundles $\mathrm{End}(TP)$ and $\mathrm{End}(TM)$. For instance, the action of $G$ on $P$ induces an action $G\times\mathrm{End}(TP)\to \mathrm{End}(TP)$ by 
	\[(g,A)\mapsto L_{g*}\circ A \circ L_{g*}^{-1}.\]
	
	\begin{definition}[Invariant Endomorphism]
		Let $T\in \Gamma(\mathrm{End}(TP))$. It is called a $G$-invariant endomorphism if $L_{g*}\circ T=T\circ L_{g*}$ for all $g\in G$. 
	\end{definition}

	\subsection{Invariant Connections and Wang's Classification Theorem}
	In this section, we present Wang's classification theorem of invariant connections on homogeneous bundles, which classifies them in terms of linear maps between Lie algebras. Throughout this section, we assume once and for all that $P$ is a homogeneous $H$-bundle over $M=G/K$, where $\lambda: K \to H$ is a Lie group homomorphism. Let $x_0\in M$ be a chosen base point. 
	
	\begin{definition}
		Let $\xi\in \mfg$. Denote by $\xi^*\in TP$ the vector field
		\[\xi^*_p:=\frac{\mrd}{\mrd t}\Bigg|_{t=0}\exp(t\xi)\, p, \quad p\in P.\]
	\end{definition}
	
	We begin with the following lemma.
	
	\begin{lemma}
		Let $P$ be a homogeneous bundle over $M=G/K$. Suppose $\Phi\in \Omega^1(P;\mfh)$ is an invariant right-equivariant $\mfh$-valued one-form of type $\Ad$. Namely, it satisfies
		\[R_h^*\Phi = \Ad_{h^{-1}}\Phi, \quad L_g^*\Phi = \Phi,\]
		for all $h\in H$ and $g\in G$. Let $p_0\in P_{x_0}$ be a fixed point. Define the linear map $\Lambda:\mfg\to\mfh$ by
		\[\Lambda(\xi):=\Phi_{p_0}(\xi_{p_0}^*), \quad \xi\in \mfg.\]
		Then the following are true.
		\begin{enumerate}
			\item $\Lambda\circ \Ad_k(\xi)=\Ad_{\lambda(k)}\circ \Lambda$ for all $k\in K$.
			\item $\Lambda(\kappa) = \Phi(\lambda_{*,e}(\kappa)\,_{p_0}^\#)$ for all $\kappa\in\mfk$.
		\end{enumerate}
		\label{lemma for Lambda}
	\end{lemma}
	
	\begin{proof}~
		\begin{enumerate}
			\item Let $\xi\in\mfg$ and $k\in K$. Then 
			\begin{align*}
				\Lambda\circ \Ad_k(\xi)
				&=\Phi_{p_0}((\Ad_k\xi)_{p_0}^*)\\
				&=\Phi_{p_0}\left(\frac{\mrd}{\mrd t}\Big|_{t=0}k\exp(t\xi)k^{-1}p_0\right)\\
				&=\Phi_{p_0}\left(\frac{\mrd}{\mrd t}\Big|_{t=0}k\exp(t\xi)p_0\lambda(k)^{-1}\right) && (\text{by the definition of $\lambda$})\\ 
				&=\Phi_{p_0}\left(L_{k*}\circ R_{\lambda(k)^{-1}*}(\xi_{p_0}^*)\right)\\
				&=\Phi_{k^{-1}p_0}(R_{\lambda(k)^{-1}*}(\xi_{p_0}^*)) && (\text{$\Phi$ is invariant})\\
				&=\Ad_{\lambda(k)}\circ \Phi_{p_0}(\xi_{p_0}^*) && (\text{$\Phi$ is right-equivariant})\\
				&=\Ad_{\lambda(k)}\circ \Lambda(\xi).
			\end{align*}
			\item Let $\kappa\in \mfk$. Then
			\begin{align*}
				\Lambda(\kappa)
				&=\Phi_{p_0}(\kappa_{p_0}^*)\\
				&=\Phi_{p_0}\left(\frac{\mrd}{\mrd t}\Big|_{t=0}\exp(t\kappa)p_0\right)\\
				&=\Phi_{p_0}\left(\frac{\mrd}{\mrd}\Big|_{t=0}p_0\lambda(\exp(t\kappa))\right)\\
				&=\Phi_{p_0}(\lambda_{*}(\kappa)_{p_0}^\#).
			\end{align*}
		\end{enumerate}
	\end{proof}
	
	The next lemma will be needed in proving Wang's classification theorem.

	\begin{lemma}
		Every vector belongs to $T_p P$ can be written as $\xi^*_p+\eta_p^\#$ for some $\xi\in \mfg$ and $\eta\in \mfh$.
	\end{lemma}
	
	\begin{proof}
		Let $p_0\in P_{x_0}$ and $p\in P$. Since $P$ is a homogeneous bundle, there exist $g\in G, h\in H$ such that $p_0=gph$. Therefore, every vector at $p$ can be translated to $p_0$ by $L_g$ and $R_h$. Moreover, $L_g$ and $R_h$ preserve fundamental vector fields induced by $G$- and $H$-action. For instance, $L_{g*}(\xi^*_p)=(\Ad_{g}\xi)_{gp}^*$ and $R_h(\xi^*_p)=\xi^*_{ph}$ for every $g\in G, h\in H$ and $\xi\in \mfh$. Hence, it suffices to prove the statement at the point $p_0$. Let $X\in T_{p_0}P$. Since $r_{x_0*}:\mfg\to T_{x_0} M$ is surjective, there exists $\xi\in \mfg$ such that $r_{x_0*}(\xi)=\pi_{*,p_0}(X)$. It follows that 
		\[\pi_{*,p_0}(X-\xi^*_{p_0})=\pi_{*,p_0}X-r_{x_0*}(\xi)=0.\]
		In other words, $X-\xi^*_{p_0}\in V_{p_0} P$. Thus, there exists $\eta\in \mfh$ such that $X-\xi^*_{p_0}=\eta_{p_0}^\#$.
	\end{proof}
	
	\begin{proposition}
		Let $P\to M$ be a homogeneous $H$-bundle and $\theta: TP\to \mfh$ an invariant connection on $P$. Define $\Lambda: \mfg \to \mfh$ by 
		\[\Lambda(\xi):=\theta_{p_0}(\xi^*_{p_0}), \quad \xi\in \mfg,\]
		where $p_0\in P_{x_0}$ is fixed and 
		\[\xi^*_{p_0}:=\frac{\mrd}{\mrd t}\Bigg|_{t=0} \exp(t\xi)\,p_0.\]
		Then the following hold. 
		\begin{enumerate}
			\item For all $k\in K$, $\Lambda\circ \Ad_k = \Ad_{\lambda(k)}\circ \Lambda$.
			\item For all $\kappa\in \mfk$,
			$\Lambda(\kappa)=\lambda_{*,e}(\kappa)$.
			\item The curvature $F$ of $\theta$ satisfies 
			\[F_{p_0}(\xi^*_{p_0},\zeta^*_{p_0})=[\Lambda(\xi),\Lambda(\zeta)]-\Lambda([\xi,\zeta])\quad \xi,\zeta\in \mfg.\]
		\end{enumerate}
		\label{wang classification theorem easier direction}
	\end{proposition}
	
	\begin{proof}
		Recall that the connection one-form $\theta$ is right-equivariant of type $\Ad$. Therefore, (1) and (2) follow from Lemma \ref{lemma for Lambda}.
		To prove (3), let $\xi,\zeta\in \mfg$. Since $\theta$ is invariant, we see that 
		\[\mathcal{L}_{\xi^*}\theta=\frac{\mrd}{\mrd t}\Bigg|_{t=0}(L_{\exp t\xi}^*\theta)=\frac{\mrd}{\mrd t}\Bigg|_{t=0}\theta = 0.\]
		Therefore,
		\[
		\xi^*_{p_0}(\theta(\zeta^*))=(\mathcal{L}_{\xi}\theta)_{p_0}(\zeta^*_{p_0})+\theta_{p_0}([\xi^*,\zeta^*])=0+\theta_{p_0}([\zeta,\xi]^*)=-\Lambda([\xi,\zeta]).
		\]
		As a result, the curvature two-form is given by
		\begin{align*}
			F(\xi^*_{p_0},\zeta^*_{p_0})
			&= \mrd\theta(\xi^*_{p_0},\zeta^*_{p_0})+[\theta(\xi^*_{p_0}),\theta(\zeta^*_{p_0})]\\
			&=-\Lambda([\xi,\zeta])+\Lambda([\zeta,\xi])-\theta_{p_0}([\xi^*_{p_0},\zeta^*_{p_0}])+[\Lambda(\xi),\Lambda(\zeta)]\\
			&=-2\Lambda([\xi,\zeta])-\theta_{p_0}([\zeta,\xi]_{p_0}^*)+[\Lambda(\xi),\Lambda(\zeta)]\\
			&=-2\Lambda([\xi,\zeta])+\Lambda([\xi,\zeta])+[\Lambda(\xi),\Lambda(\zeta)]\\
			&=[\Lambda(\xi),\Lambda(\zeta)]-\Lambda([\xi,\zeta]),
		\end{align*}
		where we used the fact that $[\xi^*,\zeta^*]=-[\xi,\zeta]^*$ since $G$ acts on $P$ from the left.
	\end{proof}
	
	Wang's classification theorem gives the converse.
	
	\begin{theorem}[Wang's Classification Theorem]
		Let $P=G\times_\lambda H\to M=G/K$ be a principal $H$-bundle. There is a one-to-one correspondence between the space of invariant connections on $P$ and the set of linear maps $\Lambda: \mfg\to \mfh$ such that $\Lambda\circ \Ad_k=\Ad_{\lambda(k)}\circ\Lambda$ and $\Lambda(\kappa)=\lambda_{*,e}(\kappa)$ for all $k\in K,\kappa\in \mfk$.
		\label{Wang classification theorem}
	\end{theorem}
	
	\begin{proof}
		One direction is proved in Proposition \ref{wang classification theorem easier direction}. For the converse, let $\Lambda: \mfg\to \mfk$ be a linear map satisfying the assumptions. Let $p\in P$, $X\in T_p P$ and $g\in G, h\in H$ such that $p_0=gph$. By the lemma above, there exist $\xi\in \mfg, a\in \mfh$ such that $L_{g*}\circ R_{h*}(X)=\xi^*_{p_0}+a_{p_0}^\#$. Define $\theta_p:T_p P \to \mfh$ by 
		\[\theta(X):=\Ad_h(\Lambda(\xi)+a).\]
		We first prove that $\theta(X)$ is independent of the choice of $\xi$ and $a$. Suppose
		\[
		\xi_{p_0}^* + a_{p_0}^\# = \zeta_{p_0}^* + b_{p_0}^\#,
		\quad\text{where } \zeta \in \mathfrak{g} \text{ and } b \in \mathfrak{h},
		\]
		so that $\xi_{p_0}^* - \zeta_{p_0}^* = b_{p_0}^\# - a_{p_0}^\#$. Next, observe that $\xi-\zeta\in \mfk$. This is because
		\[\pi_{*,p_0}(\xi_{p_0}^*-\zeta_{p_0}^*)=\pi_*(b^\sharp-a^\sharp)=0,\]
		and 
		\[\pi_{*,p_0}(\xi_{p_0}^*)=\frac{\mrd}{\mrd t}\Big|_{t=0}\pi(\exp(t\xi)p_0)=\frac{\mrd}{\mrd t}\Big|_{t=0}\exp(t\xi)x_0=r_{x_0*}(\xi).\]
		Therefore, $\xi-\zeta\in \ker r_{x_0*}=\mfk$. Next, for any $\kappa\in\mfk$, we have
		\[\kappa_{p_0}^*=\frac{\mrd}{\mrd t}\Big|_{t=0}\exp(t\kappa)p_0=\frac{\mrd}{\mrd t}\Big|_{t=0}p_0\lambda(\exp(t\kappa))=(\lambda_*\kappa)_{p_0}^\#.\]
		Applying this formula for $\kappa = \xi-\zeta$, it follows that 
		\[
		\lambda_*(\xi - \zeta) = b - a.
		\]
		By assumption, $\Lambda|_\mfk=\lambda_*$. Thus,
		\[
		\Lambda(\xi) - \Lambda(\zeta) = \Lambda(\xi - \zeta) = b-a.
		\]
		Hence, $\Lambda(\xi) + a = \Lambda(\zeta) + b$.
		
		One can show that that $\theta(\xi)$ is independent of the choice of $g$ and $h$. 
		Finally, we check that is an invariant connection. It is obvious that $\theta_p(\eta_p^\#)=\eta$ for $p\in P$ and $\eta\in \mfh$. Let $h'\in H$ and $X=L_{g^{-1}*}\circ R_{h^{-1}*}(\xi^*_{p_0}+\eta_{p_0}^\#)$. Since $g(ph')h'^{-1}h=p_0$ and $L_{g*}\circ R_{h'^{-1}h*} (R_{h'*}X)= \xi_{p_0}^*+a_{p_0}^\# $, by the definition of $\theta_{ph'}$, it follows that
		\begin{align*}
			(R_{h'}^*\theta)_p(X)
			&=\theta_{ph'}(R_{h'}X)\\
			&=\Ad_{h'^{-1}h} (\Lambda(\xi)+a)\\
			&=\Ad_{h'^{-1}} \theta_p(X).
		\end{align*}
		It remains to show that it is invariant. Let $g'\in G$ and $X=L_{g^{-1}*}\circ R_{h^{-1}*}(\xi^*_{p_0}+\eta_{p_0}^\#)$. 
		Since $gg'^{-1}(g'p)h=p_0$ and $L_{gg'^{-1}*}\circ R_{h*}(L_{g'*}X)=L_{g*}\circ R_{h*}(X)=\xi_{p_0}^*+a_{p_0}^\#$, by the definition of $\theta_{g'p}$, we have
		\begin{align*}
			(L_{g'}^*\theta)_p(X)
			&=\theta_{g'p}(L_{g'*}X)\\
			&=\Ad_h(\Lambda(\xi)+a)\\
			&=\theta_p(X).
		\end{align*}
	\end{proof}
	
	As a corollary, we show that every homogeneous bundle over reductive homogeneous spaces admits an invariant connection. Before that, we define reductive homogeneous spaces.
	
	\begin{definition}[Reductive Homogeneous Space]
		Let $M=G/K$ be a homogeneous space. It is called reductive if there is an $\Ad_K$-invariant subspace $\mfm$ of $\mfg$ complementary to $\mfk$. In other words, there exists a vector subspace $\mfm$ of $\mfg$ such that $\mfg = \mfk\oplus \mfm$ and 
		\[\Ad_k(\mfm)\subseteq \mfm,\]
		for every $k\in K$.
	\end{definition}
	
	\begin{corollary}
		Let $P=G\times_\lambda H$ be a homogeneous $H$-bundle over a reductive homogeneous space $M$. Then $P$ admits an invariant connection.
	\end{corollary}
	
	\begin{proof}
		Let $\mfm$ be an $\Ad_K$-invariant subspace of $\mfg$ complementary to $\mfk$. Define $\Lambda:\mfg\to\mfh$ by
		
		\begin{equation*}
			\Lambda(\xi):=
			\begin{cases}
				\lambda_{*,e}(\xi) & \xi\in \mfk,\\
				0 & \xi\in\mfm.
			\end{cases}
		\end{equation*}
		By definition of $\Lambda$, $\Lambda|_\mfk = \lambda_{*,e}$. Let $k\in K$ and $\xi\in\mfg$. If $\xi \in \mfk$, then
		\[\Lambda(\Ad_k \xi)=\lambda_{*,e}(\Ad_k\xi)=\Ad_{\lambda(k)}\lambda_{*,e}(\xi)=\Ad_{\lambda(k)}\Lambda(\xi).\]
		On the other hand, if $\xi\in \mfm$, then $\Ad_k\xi\in \mfm$. Hence, 
		\[\Lambda(\Ad_k(\xi))=0=\Ad_{\lambda(k)}(\Lambda(\xi)).\]
		Therefore, by the Wang's classification, there exists an invariant connection on $P$.
	\end{proof}
	
	\begin{definition}
		The canonical invariant connection on a homogeneous $H$-bundle over a reductive homogeneous space $M$ is the invariant connection corresponding to the linear map $\Lambda:\mfg\to\mfh$ defined by
		\begin{equation*}
			\Lambda(\xi):=
			\begin{cases}
				\lambda_{*,e}(\xi) & \xi\in \mfk,\\
				0 & \xi\in\mfm.
			\end{cases}
		\end{equation*}
	\end{definition}
	
	A closely related classification in terms of tensorial forms is given in the next section.

	\subsection{Invariant Connections and Tensorial Forms}
	In this section, we present another classification of invariant connections using the one-to-one correspondence between the space of tensorial forms and the space of connections (Theorem \ref{space of connections as affine space}). Recall that we have a map $r_{x_0}:G\to M$ sending $g\in G$ to $g\cdot x_0\in M$ and its differential $r_{x_0*,e}:\mfg\to T_{x_0}M$ restricted to $\mfg/\mfk$ is an isomorphism.

	\begin{theorem}
		Let $P=G\times_\lambda H\to M$ be a homogeneous $H$-bundle. There is a one-to-one correspondence between the space of invariant tensorial one forms on $P$ and the subset of $\mathrm{Hom}(\mfg,\mfh)$ consisting of linear maps $\phi:\mfg\to\mfh$ such that
		\begin{enumerate}
			\item $\phi|_\mfk=0$.
			\item $\phi\circ \Ad_k = \Ad_{\lambda(k)}\circ \phi$ for all $k\in K$.
		\end{enumerate}
		\label{one-to-one correspondence between tensorial forms and phi,psi}
	\end{theorem}

	\begin{proof}
		Given an invariant tensorial form $\Phi:TP\to\mfh$, define 
		\[\phi=\Phi_{p_0}\circ r_{p_0*},\]
		where $r_{p_0}:G\to P$ sends $g$ to $gp_0$. By Lemma \ref{lemma for Lambda}, $\phi:\mfg\to\mfh$ is the desired linear map. Conversely, if $\phi:\mfg\to\mfh$ is a linear map satisfying conditions (1) and (2), by fixing $p_0\in P_{x_0}$, we define $\Phi$ by 
		\[\Phi_{p_0}(\xi_{p_0}^*)=\phi(\xi),\quad \xi\in \mfg,\]
		and extend to a global one-form on $P$ by $G$-invariance and $H$-equivariance. The proof is similar to the proof of Wang's classification theorem (Theorem \ref{Wang classification theorem}).
	\end{proof}

	\begin{remark}~
		The map $\phi:\mfg\to\mfk$ is not necessarily a Lie algebra homomorphism.
		\label{lambda iff lambda'}
	\end{remark}

	\begin{theorem}
		Let $\lambda:K\to H$ be a Lie group homomorphism and
		\[P=G\times_\lambda H\to M.\] 
		Suppose there exists an invariant connection on $P$. Then there exists a one-to-one correspondence between the space of invariant connections on $P$ and the subset of $\mathrm{Hom}(\mfg,\mfh)$ consisting of linear maps $\phi:\mfg\to\mfh$ satisfying the following conditions:
		\begin{enumerate}
			\item $\phi|_\mfk=0$.
			\item $\phi$ satisfies the equation
			\begin{equation}
				\phi\circ \Ad_k = \Ad_{\lambda(k)}\circ \phi \text{ for all $k\in K$}.
				\label{Lambda circ Ad_k}
			\end{equation}
		\end{enumerate}
		Moreover, if $K$ is connected, then condition \eqref{Lambda circ Ad_k} may be replaced by:
		\begin{equation}
			\phi\circ \on{ad}_\kappa = \on{ad}_{\lambda_{*,e}(\kappa)}\circ \phi \text{ for all $\kappa\in \mfk$},
			\label{Lambda circ ad_k}
		\end{equation}
		where $\on{ad}:\mfh\to \mathfrak{gl}(\mfh)$ denotes the adjoint representation of $\mfh$ (and $\mfk$), defined by $\on{ad}_{\eta_1}(\eta_2)=[\eta_1,\eta_2]$ for all $\eta_1,\eta_2\in \mfh$.
		\label{classification of invariant connections (Gx_lambda H)}
	\end{theorem}
	\begin{proof}
		Let $\theta^0$ be a fixed invariant connection on $P$. Since the space of invariant connections is an affine space modeled on the space of invariant tensorial one-forms of type $\Ad$, it suffices to show that there is a one-to-one correspondence between the space of invariant tensorial forms and the given subset of $\mathrm{Hom}(\mfg,\mfh)$. This is proved previously in Theorem \ref{one-to-one correspondence between tensorial forms and phi,psi}.
		
		If in addition $K$ is connected, then $K$ is generated by elements of the form $\exp \kappa$ where $\kappa\in\mfk$. If $\phi$ satisfies Equation \eqref{Lambda circ ad_k}, then for all $\kappa\in \mfk$,
		\begin{align*}
			\phi\circ \Ad ({\exp \kappa})
			&= \phi\circ \exp(\on{ad}(\kappa))\\
			&= \exp(\phi(\on{ad}(\kappa)))\\
			&= \exp(\on{ad}_{\lambda_*(\kappa)}\circ \phi)\\
			&= \Ad_{\exp(\lambda_*(\kappa))}\circ \phi\\
			&= \Ad_{\lambda(\exp (\kappa))}\circ \phi.
		\end{align*}
		Therefore, Equation \eqref{Lambda circ Ad_k} is satisfied by elements of $K$ of the forms $\exp\kappa$ for any $\kappa\in \mfk$. Since $\phi$ and $\Ad$ are group homomorphisms, and $K$ is generated by elements of the form $\exp\kappa$, Equation \eqref{Lambda circ Ad_k} holds for every $k\in K$.
	\end{proof}
	

	The correspondence may be simplified. Since $\phi|_\mfk=0$, the space of such maps $\phi$ is naturally identified with the subset
	\[\left\{\phi\in \mathrm{Hom}\left(\frac{\mfg}{\mfk},\mfk\right): \phi\circ \Ad_k = \Ad_{\lambda(k)}\circ \phi \text{ for all $k\in K$}\right\}.\]

	\subsection{Invariant Complexions}
	In this section, we classify invariant complexions in terms of linear maps, following an approach analogous to that used for invariant connections. Recall that a complexion $J$ (which is an endomorphism) is invariant if $J\circ L_{g*}=L_{g*}\circ J$ for all $g\in G$. We assume in this section that $H$ is a complex Lie group and $j$ is an invariant almost complex structure on $M$.
	
	\begin{lemma}
		Let $J^0$ and $J$ be two complexions on a homogeneous $H$-bundle related by $\Psi \in \Omega_{\Ad}^{0,1}(P,J^0;\mfh)$, i.e., 
		$J=J^0+(\Psi(\cdot))^\#$. Suppose that $J^0$ is invariant. Then $J$ is invariant if and only if $\Psi$ is invariant. 
	\end{lemma}
	
	\begin{proof}
		Since $J^0$ is invariant, $J$ is invariant if and only if $L_{g*}(\Psi(X))^\#=(\Psi(L_{g*}(X)))^\#$ for all $g\in G$ and $X\in TP$. This is equivalent to $L_g^*\Psi=\Psi$ for all $g\in G$.
	\end{proof}
	
	Therefore, by the one-to-one correspondence between tensorial forms and the subset of linear maps between Lie algebras (Theorem \ref{one-to-one correspondence between tensorial forms and phi,psi}), we get the following result.
	
	\begin{theorem} 
		Let $\lambda: K \to H$ be a Lie group homomorphism and 
		\[
		P = G \times_\lambda H \to (M, j).
		\]
		Suppose there exists an invariant complexion on $P$. Then there is a one-to-one correspondence between the space of invariant complexions on $P$ and the subset of $\mathrm{Hom}(\mathfrak{g}_\mathbb{C}, \mathfrak{h})$ consisting of linear maps 
		\[
		\psi: \mathfrak{g}_\mathbb{C} \to (\mathfrak{h}, i_H)
		\]
		satisfying the following conditions:
		\begin{enumerate}
			\item $\psi|_{\mathfrak{k}_\mathbb{C}} = 0$.
			\item $\psi \circ \mathrm{Ad}_k = \mathrm{Ad}_{\lambda(k)} \circ \psi$ for all $k \in K$.
			\item The restriction $\psi|_{\mathfrak{g}_\mathbb{C}/\mathfrak{k}_\mathbb{C}}$ satisfies
			\[
			\psi \circ j_{\mathfrak{g}} = -i \psi,
			\]
			where $j_{\mathfrak{g}} : \mathfrak{g}/\mathfrak{k} \to \mathfrak{g}/\mathfrak{k}$ is the almost complex structure on $\mathfrak{g}/\mathfrak{k}$ induced by $j_{x_0}$ via the isomorphism $r_{x_0*} : \mathfrak{g}/\mathfrak{k} \to T_{x_0}M$, i.e.,
			\[
			j_{\mathfrak{g}} = r_{x_0*}^{-1} \circ j_{x_0} \circ r_{x_0*}.
			\]
			If this is the case, we say $\psi$ is anti-linear with respect to $j_\mfg$.
		\end{enumerate}
		If $K$ is connected, then condition (2) may be replaced by:
		\[
		\psi \circ \mathrm{ad}_\kappa = \mathrm{ad}_{\lambda_{*,e}(\kappa)} \circ \psi \quad \text{for all } \kappa \in \mathfrak{k}.
		\]
		\label{classification of invariant complexions: G times_lambda H}
	\end{theorem}
	\vspace{-1cm}
	\begin{proof}
		Let $J^0$ be a fixed complexion on $P$ and $J$ another invariant complexion. Suppose $\Psi$ is the $(0,1)$-tensorial form (with respect to $J^0$) such that $J=J^0+(\Psi(\cdot))^\#$ (Theorem \ref{space of complexions as an affine space}). It remains to show that the corresponding linear map $\psi$ of a $(0,1)$-tensorial form $\Psi$ with respect to $J^0$ (Theorem \ref{one-to-one correspondence between tensorial forms and phi,psi}) is antilinear with respect to $j_{\mfg}$ if and only if $\Psi$ is antilinear with respect to $J^0$.

		Let $p_0\in P_{x_0}$ and $r_{p_0}:G\to P$ be the map sending $g$ to $gp$. We see that
		\begin{align*}
			\psi\circ j_\mfg&=-i\psi\\
			\iff 
			\Psi_{p_0}\circ r_{p_0*}\circ j_\mfg &= -i \Psi_{p_0}\circ r_{p_0*}\\
			\iff 
			\Psi_{p_0} \circ J_{p_0} &= -i\Psi.
		\end{align*}
		The case where $K$ is connected can be proved analogously to the proof of Theorem \ref{classification of invariant connections (Gx_lambda H)}.
	\end{proof}

	The correspondence may be simplified. Since $\psi|_\mfk=0$, the space of such maps $\psi$ is naturally identified with the subset
	\[
	\left\{ \psi \in \mathrm{Hom}\left( \frac{\mathfrak{g}_\mathbb{C}}{\mathfrak{k}_\mathbb{C}}, \mathfrak{h} \right) 
	:
	\psi \circ \mathrm{Ad}_k = \mathrm{Ad}_{\lambda(k)} \circ \psi \text{ for all } k \in K, \quad \psi \circ j_{\mathfrak{g}} = -i\psi
	\right\}.
	\]

	\subsection{Integrable Complexions on Homogeneous Bundles over Symmetric Spaces}
	In this section, we will show that the canonical complexion (the complexion induced by the canonical connection) is integrable. After that, we determine the condition for an arbitrary complexion to be integrable when the base $M$ is a symmetric space.
	
	We first assume that $M$ is a reductive homogeneous space equipped with an (integrable) invariant complex structure $j$, and $P\to M$ is a homogeneous $H$-bundle with $H$ being a complex Lie group. Let $\mfg=\mfk\oplus\mfm$ be an $\Ad_K$-invariant splitting. Recall that the canonical connection on $P$ is the invariant connection corresponding to the linear map $\Lambda:\mfg\to\mfh$ given by
	\begin{equation*}
		\Lambda(\xi):=
		\begin{cases}
			\lambda_{*,e}(\xi) & \xi\in \mfk,\\
			0 & \xi\in\mfm.
		\end{cases}
	\end{equation*}
	
	\begin{lemma}
		If $\kappa\in \mfk$, then 
		\[\kappa_{p_0}^*=(\lambda_*\kappa)_{p_0}^\#.\]
		If $\xi\in \mfm$, then 
		\[\xi_{p_0}^*=\widetilde{(\xi_{x_0}^*)}_{p_0}\]
		with respect to the canonical connection.
		\label{star: vertical and horizontal}
	\end{lemma}
	
	\begin{proof}
		By the definition of $\lambda:K\to H$, if $\kappa\in \mfk$, then
		\begin{align*}
			\kappa_{p_0}^*
			&=\frac{d}{dt}\Bigg|_{t=0} \exp t\kappa \cdot p_0\\
			&=\frac{d}{dt}\Bigg|_{t=0} p_0\cdot\lambda(\exp t\kappa)\\
			&=(\lambda_*\kappa)_{p_0}^\#.
		\end{align*}
		Now suppose $\xi\in \mfm$. The horizontal part of $\xi_{p_0}^*$ is $\widetilde{(\xi_{x_0}^*)}_{p_0}$. On the other hand, its vertical component is
		\[(\theta(\xi_{p_0}^*))^\#=(\Lambda(\xi))^\#=0.\]
	\end{proof}
	
	Denote by $J$ the canonical complexion. Via the isomorphisms $A_{p_0}\cong T_{x_0}M\cong \mfm$, we define $j_\mfm:\mfm\to\mfm$ so that the following diagram is commutative:
	
	\adjustbox{scale=1.2,center}{
		\begin{tikzcd}
			A_{p_0} \arrow[r, "J_{p_0}"] \arrow[d, "{\pi_{*,p_0}}"']                      & A_{p_0} \arrow[d, "{\pi_{*,p_0}}"]       \\
			T_{x_0}M \arrow[r, "j_{x_0}"]                                                 & T_{x_0}M                                 \\
			\mathfrak{m} \arrow[r, "j_{\mathfrak{m}}", dashed] \arrow[u, "{r_{x_0 *,e}}"] & \mathfrak{m} \arrow[u, "{r_{x_0 *,e}}"']
		\end{tikzcd}
	}
	
	In other words, if $\xi\in \mfm$, then
	\[(j_m\xi)_{x_0}^*=j_{x_0}(\xi_{x_0}^*), \quad (j_m\xi)_{p_0}^*=J_{p_0}(\xi_{p_0}^*).\]
	
	\begin{lemma}
		Let $P=G\times_\lambda H$ and $\xi,\zeta\in \mfm$. Then 
		\[\lambda_*(P_\mfk[\xi,j_m\zeta]) = i_H\lambda_*(P_\mfk[\xi,\zeta]),\]
		where $P_\mfk:\mfg\to\mfk$ is the projection onto $\mfk$ with respect to the decomposition $\mfg=\mfk\oplus\mfm$.
		\label{lemma for integrability of canonical complexion}
	\end{lemma}
	
	\begin{proof}
		Let $\xi,\zeta\in\mfm$. Since $J$ is induced by an invariant connection, it follows that $J$ is invariant and hence
		\[\mathcal{L}_{\xi^*}J=0.\]
		Therefore,
		\[[\xi_{p_0}^*,J(\zeta_{p_0}^*)]=J[\xi_{p_0}^*,\zeta_{p_0}^*].\]
		Using the map $j_\mfm:\mfm\to\mfm$, this is equivalent to
		\[[\xi,j_\mfm\zeta]_{p_0}^*=J([\xi,\zeta]_{p_0}^*).\]
		Let $P_\mfk\xi$ and $P_\mfm\xi$ be the $\mfk$- and $\mfm$-components of $\xi$ respectively. Using Lemma \ref{star: vertical and horizontal}, the left hand side the equation above is
		\begin{align*}
			(P_\mfk[\xi,j_m\zeta])_{p_0}^*+(P_\mfm[\xi,j_m\zeta])_{p_0}^*
			&=(\lambda_*(P_\mfk[\xi,j_m\zeta]))_{p_0}^\#+(\pi|_A)_{*}^{-1}{((P_\mfm[\xi,j_m\zeta])_{x_0}^*}),
		\end{align*}
		while the right hand side of the equation is
		\begin{align*}
			J(P_\mfk[\xi,\zeta])_{p_0}^*+J(P_\mfm[\xi,\zeta])_{p_0}^*
			&=J((\lambda_*P_\mfk[\xi,\zeta])_{p_0}^\#)+J((\pi|_A)_*^{-1}(P_\mfm[\xi,\zeta])_{x_0}^*))\\
			&=(i_H P_\mfk[\xi,\zeta])_{p_0}^\#+(\pi|_A)_*^{-1}((j_\mfm P_m[\xi,\zeta])_{x_0}^*).
		\end{align*}
		Comparing the vertical components of both sides, we get the desired result.
	\end{proof}
	
	\begin{theorem}
		Suppose $M$ is reductive homogeneous equipped with the complex structure $j$, and $P$ a homogeneous $H$-bundle over $M$. Then the canonical complexion on $P$ is integrable.
		\label{canonical complexion is integrable}
	\end{theorem}
	\begin{proof}
		Recall that the canonical almost complex structure $J$ is integrable if and only if the $(0,2)$-part $F^{0,2}$ of the curvature form $F$ vanishes (Corollary \ref{integrability condition and F_A}). By Proposition \ref{wang classification theorem easier direction}, the curvature form $F$ is related to $\Lambda$ by
		\[F(\xi_{p_0}^*,\zeta_{p_0}^*)=[\Lambda(\xi),\Lambda(\zeta)]-\Lambda[\xi,\zeta],\quad \xi,\zeta\in \mfm.\]
		By Lemma \ref{star: vertical and horizontal} , if either $\xi$ or $\zeta$ is an element of $\mfk$, then $F(\xi_{p_0}^*,\zeta_{p_0}^*)$ vanishes. Hence, it suffices to show 
		\[F^{0,2}(\xi_{p_0}^*,\zeta_{p_0}^*)=0\]
		whenever $\xi$ and $\zeta$ are elements of $\mfm$. Recall that $\Lambda|_\mfm=0$ and $\Lambda|_\mfk=\lambda_*$. Thus, if $\xi$ and $\zeta$ are from $\mfm$, when evaluated at $p_0$, we have
		\[F(\xi^*,\zeta^*)=[\Lambda(\xi),\Lambda(\zeta)]-\Lambda[\xi,\zeta]=-\Lambda(P_\mfk[\xi,\zeta])-\Lambda(P_\mfm[\xi,\zeta])=-\lambda_*(P_\mfk[\xi,\zeta]).\]
		Therefore, it follows from Lemma \ref{lemma for integrability of canonical complexion} that
		\begin{align*}
			F^{0,2}(\xi^*,\zeta^*)
			&=F(\xi^*,\zeta^*)-F(J\xi^*,J\zeta^*) + i_H\left(F(\xi^*,J\zeta^*)+F(J\xi^*,\zeta^*)\right)\\
			&={{-\lambda_*(P_\mfk[\xi,\zeta])}}+\lambda_*(P_\mfk[j_m\xi,j_m\zeta]){-i_H\lambda_*(P_\mfk[\xi,j_m\zeta])}-i_H\lambda_*(P_\mfk[j_m\xi,\zeta])\\
			&=0.
		\end{align*}
	\end{proof}
	
	In the following discussion, we examine the conditions for any complexion to be integrable. Now we assume $M$ is a symmetric space.
	
	\begin{definition}[Symmetric Spaces]
		A \emph{symmetric space} is a triple $(G,H,\sigma)$ consisting of a connected Lie group $G$, a closed subgroup $K$ of $G$ and an involutive automorphism $\sigma$ of $G$ such that $K$ lies between $G_\sigma$ and the identity component of $G_\sigma$, where $G_\sigma$ denotes the closed subgroup of $G$
		\[G_\sigma=\{g\in G: \sigma(g)=g\}.\]
		\label{definition: symmetric spaces}
	\end{definition}
	\vspace{-1cm}
	Given a symmetric space $(G,H,\sigma)$, the coset manifold $M=G/H$ is a homogeneous space. In fact, it is reductive homogeneous. First note that since $\sigma^2=\operatorname{id}_G$, it follows that $\sigma_*$ has only two eigenvalues which are $1$ and $-1$.
	
	\begin{proposition}
		Let $\mfm\subset\mfg$ be the eigenspace of $\sigma_*$ with eigenvalue $-1$. Then 
		\begin{enumerate}
			\item The subspace $\mfk$ is the eigenspace of $\sigma_*$ with eigenvalue $1$. In particular,
			\[\mfg=\mfk\oplus\mfm.\]
			\item The subspaces $\mfg,\mfk,\mfm$ satisfy
			\[[\mfk,\mfk]\subseteq\mfk, \quad [\mfk,\mfm]\subseteq\mfm,\quad [\mfm,\mfm]\subseteq \mfk.\]
		\end{enumerate}
	\end{proposition}
	
	\begin{proof}
		\begin{enumerate}
			\item Since $K$ lies between $G_\sigma$ and the identity component of $G_\sigma$, it follows that $\mfk$ is the eigenspace of $\sigma_*$ with eigenvalue 1. The splitting follows from linear algebra.
			\item Let $\kappa\in \mfk$ and $\xi\in \mfm$. Then
			\[\sigma_*([\kappa,\xi])=[\sigma_*\kappa,\sigma_*\xi]=[\kappa,-\xi]=-[\kappa,\xi].\]
			Therefore, $[\kappa,\xi]\in \mfm$. Similarly, one proves that $[\mfm,\mfm]\subseteq\mfk$.
		\end{enumerate}
	\end{proof}
	
	Now we assume $M$ is symmetric and is endowed with an invariant complex structure $j$, and $P$ is a homogeneous $H$-bundle over $M$. By the previous proposition, $M$ is reductive homogeneous. Therefore, the canonical connection $\theta^0$ on $P$ induces an integrable complexion $J^0$ on $P$ by Theorem \ref{canonical complexion is integrable}. We denote the curvature two-form of $\theta^0$ by $F^0$.

	\begin{lemma}
		Let $\theta$ be a connection on $P$ and $\Phi\in \Omega_{\Ad}^1(P;\mfh)$ be such that \[\theta=\theta^0+\Phi.\]
		Let $F$ be the curvature form of $\theta$, and $p_0\in P_{x_0}$ be fixed. For all $\xi,\zeta\in \mfm$, we have
		\[F_{p_0}(\xi^*,\zeta^*)-F_{p_0}^0(\xi^*,\zeta^*)=\frac{1}{2}[\phi,\phi](\xi,\zeta)=[\phi(\xi),\phi(\zeta)],\]
		where $\phi(\xi)=\Phi_{p_0}(\xi_{p_0}^*)$ is the linear map from $\mfg$ to $\mfh$ associated with $\Phi$.
	\end{lemma}
	
	\begin{proof}
		By Proposition \ref{curvature 2 form formula}, the two curvature forms are related by
		\[F=\mrd\theta+\frac{1}{2}[\theta,\theta]=F^0+[\theta^0,\Phi]+\mrd\Phi+\frac{1}{2}[\Phi,\Phi].\]
		If $X,Y$ are $\theta^0$-horizontal vectors, then
		\[F(X,Y)=F^0(X,Y)+\mrd\Phi(X,Y)+\frac{1}{2}[\Phi,\Phi](X,Y).\]
		We claim that for $\xi,\zeta\in \mfm$,
		\[\mrd\Phi_{p_0}(\xi^*,\zeta^*)=0.\]
		Since $\Phi$ is invariant, we have
		\[\xi_{p_0}^*(\Phi(\zeta^*))=(\mathcal{L}_{\xi^*}\Phi)(p_0)+\Phi_{p_0}[\xi^*,\zeta^*]=-\Phi_{p_0}([\xi,\zeta]_{p_0}^*)=0.\]
		The last equality follows because $[\xi,\zeta]\in \mfk$. Therefore, 
		\begin{align*}
			\mrd\Phi_{p_0}(\xi^*,\zeta^*)
			&=\xi_{p_0}^*(\Phi(\zeta^*))-\zeta_{p_0}^*(\Phi(\xi^*))-\Phi_{p_0}[\xi^*,\zeta^*]\\
			&=0.
		\end{align*}
		Hence, we obtain
		\[F(\xi_{p_0}^*,\zeta_{p_0}^*)-F^0(\xi_{p_0}^*,\zeta_{p_0}^*)=\frac{1}{2}[\Phi,\Phi](\xi_{p_0}^*,\zeta_{p_0}^*)=\frac{1}{2}[\phi,\phi](\xi,\zeta).\]
	\end{proof}
	
	\begin{theorem}
		Let $J$ be an invariant complexion on $P$ such that
		\[J=J^0+(\Psi(\cdot))^\#.\]
		Let $\psi:\mfg\to\mfh$ be the linear map corresponding to $\Psi$. In other words,
		\[
		\psi(\xi)=\Psi_{p_0}(\xi_{p_0}^*),\quad \xi\in \mfg.
		\] 
		Then $J$ is integrable if and only if 
		\[[\psi,\psi]=0.\]
		\label{integrability condition in terms of psi}
	\end{theorem}
	
	\begin{proof}
		From Theorem \ref{relation between tensorial forms of connection and complexion}, $J$ is induced by the connection $\theta=\theta^0-\frac{i}{2}\Psi$. By the previous lemma, we obtain
		\begin{align*}
			F(\xi^*,\zeta^*)
			&=F^0(\xi^*,\zeta^*)+\frac{1}{2}\left[-\frac{i}{2}\Psi,-\frac{i}{2}\Psi\right](\xi^*,\zeta^*)\\
			&=F^0(\xi^*,\zeta^*)-\frac{1}{8}\left[\Psi,\Psi\right](\xi^*,\zeta^*).
		\end{align*}
		Since $J^0$ is integrable, the $(0,2)$-part of $F^0$ vanishes. Thus,
		\begin{align*}
			F^{0,2}(\xi^*,\zeta^*)
			&=F(2\pi^{0,1}\xi^*,2\pi^{0,1}\zeta^*)\\
			&=F^0(2\pi^{0,1}\xi^*,2\pi^{0,1}\zeta^*)-\frac{1}{8}[\Psi,\Psi](2\pi^{0,1}\xi^*,2\pi^{0,1}\zeta^*)\\
			&=-\frac{1}{8}[\Psi^{0,1},\Psi^{0,1}](\xi^*,\zeta^*)\\
			&=-\frac{1}{8}[\Psi,\Psi](\xi^*,\zeta^*)\\
			&=-\frac{1}{8}[\psi,\psi](\xi,\zeta).
		\end{align*}
		Here we used the fact that $\Psi$ is $(0,1)$-form. Hence, $J$ is integrable if and only if 
		\[[\psi,\psi]=0.\]
	\end{proof}
	
	From this theorem, we obtain the following result.
	\begin{corollary}
		With the same assumptions as in Theorem \ref{integrability condition in terms of psi}, if $H$ is abelian, then every invariant complexion on $P$ is integrable.
		\label{abelian structure group integrable}
	\end{corollary}
	
	\begin{proof}
		This follows from the fact that every linear map $\psi:\mfg_\mbc\to \mfh$ satisfies 
		\[[\psi,\psi]=0\]
		as $H$ is abelian.
	\end{proof}
	
	Finally, we have the following classification theorem for integrable complexions on homogeneous bundles over symmetric space.
	
	\begin{theorem}
		Let $M=G/K$ be a symmetric space equipped with the invariant complex structure $j$, and $P=G\times_\lambda H$ a homogeneous $H$-bundle over $M$. Then there is a one-to-one correspondence between the space of integrable invariant complexions on $P$ and the subset of $\mathrm{Hom}(\mfg_\mbc,\mfh)$ consisting of linear maps 
		\[\psi:\mfg_\mbc\to (\mfh,i_H)\]
		satisfying the following conditions:
		\begin{enumerate}
			\item $\psi|_{\mfk_\mbc}=0.$
			\item $\psi\circ \Ad_k=\Ad_{\lambda(k)}\circ \psi$ for all $k\in K$.
			\item $\psi$ is anti-linear with respect to $j_\mfm$
			\[\psi\circ j_\mfm=-i\psi.\]
			\item $[\psi,\psi]=0$.
		\end{enumerate}
		\label{classification of integrable invariant complexions}
	\end{theorem}
	
	\begin{proof}
		This follows from Theorem \ref{classification of invariant complexions: G times_lambda H} and the previous theorem.
	\end{proof}
	
	\newpage
	\section{Hermitian and K\"ahler Metrics on Principal Fibre Bundles}
	
	In this section, we study Hermitian and K\"ahler metrics induced by a connection. We will show that every connection induces a Hermitian and a K\"ahler metric on a  principal fibre bundle under appropriate assumptions on the base $M$ and the Lie group $H$.
	
	In Section 5.1, we introduce the notions of Hermitian and K\"ahler manifolds in a general setting. Section 5.2 introduces connection metrics, which are Riemannian metrics induced on the principal bundle by choices of connections. We study the conditions when the principal bundle equipped with the induced complexion and connection metric is Hermitian or K\"ahler in Sections 5.2 and 5.3. Finally, in Section 5.4, the focus is restricted to K\"ahler metrics on homogeneous bundles.
	
	\subsection{Hermitian and K\"ahler Manifolds}
	In this section, we introduce basic notions about Hermitian and K\"ahler manifolds. In the following sections, we assume that $M$ is a complex manifold (with the canonical almost complex structure which is integrable). Briefly speaking, a Hermitian metric captures the geometric information of an almost complex structure, and a K\"ahler metric is a special class of Hermitian metrics such that its fundamental two-form is closed. 
	
	\begin{definition}[Hermitian Manifolds]
		Let $M$ be a complex manifold and $J$ its canonical complex structure. A Riemannian metric $g$ on $M$ (viewed as a real manifold) is called \emph{Hermitian} if 
		\[g(J\cdot,J\cdot)=g(\cdot\,,\,\cdot).\]
		In other words, $J$ is orthogonal with respect to the metric $g$. In this case, the triple $(M,J,g)$ is said to be a Hermitian manifold.
	\end{definition}
	
	Motivated by the geometric interpretation of the imaginary number as a rotation, we require that $J$, in addition to satisfying the equation $J^2=-\on{id}$, exhibits the behaviour of a rotation.
	
	\begin{example}
		The complex manifold $\mbc^n$, equipped with its canonical complex structure and the standard Euclidean metric, defines a Hermitian manifold.
	\end{example}

	\begin{definition}[Fundamental 2-forms]
		Let $(M,J,g)$ be a Hermitian manifold. The fundamental 2-form of $g$ is defined to be the real differential 2-form $\omega\in\Omega^2(M)$ given by
		\[\omega(\cdot\,,\,\cdot):=g(J\cdot,\cdot).\]
	\end{definition}
	It is alternating because $\omega(X,Y)=g(JX,Y)=g(J^2 X,JY)=-g(X,JY)=-g(JY,X)=-\omega(Y,X)$ for all $X,Y\in \Gamma(TM)$. Moreover, since $J$ is orthogonal with respect to $g$, it is orthogonal with respect to $\omega$.
	
	\begin{definition}[K\"ahler Manifold]
		A Hermitian manifold $(M,J,g)$ is said to be a \emph{K\"ahler} manifold if its fundamental 2-form $\omega$ is closed ($\mrd \omega=0$). In this case, the fundamental 2-form $\omega$ is called the K\"ahler form of $g$.
	\end{definition}
	
	Equivalently, a Hermitian manifold $(M,J,g)$ is K\"ahler if and only if $J$ is parallel with respect to the Levi-Civita connection $\nabla$ of $g$. For a proof, see \cite{andreimoroianu}.

	\subsection{Hermitian Metrics Induced by Connections}
	
	Let $(M,g)$ be a Riemannian manifold and $H$ a Lie group endowed with a Riemannian metric $(\,,\,)$ (not necessarily complex).

	\begin{definition}[Connection Metrics]
		Let $A$ be a connection on a principal $H$- bundle $P$ and $\theta$ be its connection one-form. Define the Riemannian metric $g_A:TP\times TP\to \mbr$ on $P$ by
		\begin{equation}
			(g_A)_p(X,Y)=g(\pi_*X,\pi_*Y)_x+(\theta_p(X),\theta_p(Y)),
			\label{induced metrics}
		\end{equation}
		where $x=\pi(p)$. The Riemannian metric $g$ is called the metric induced by the connection $\theta$.
	\end{definition}
	
	Let $(M,j,g)$ be a Hermitian manifold and $P$ be a principal $H$-bundle over $M$.
	\begin{proposition}
		Suppose that $A$ is a connection on $P$ and $g_A$ is its induced metric given by \eqref{induced metrics}. Then
		\begin{enumerate}
			\item The vertical subbundle and the horizontal subbundle are orthogonal with respect to $g_A$.
			\item The map $\pi:(P,g_A)\to(M,g)$ is a Riemannian submersion, namely, for every $p\in P$, the map $\pi_{*,p}|_{A_p}:A_p\to T_pM$ is an isometry.
		\end{enumerate}
	\end{proposition}
	
	\begin{proof}~
			(1) Let $X\in \Gamma(VP)$ and $Y\in \Gamma(A)$. Then
			\(g_A(X,Y)=g({\pi_*X},\pi_*Y)+(\theta(X),{\theta(Y)})=0.\)\\
			(2) Let $X,Y\in \Gamma(A)$. It follows that
			\(g_A(X,Y)=g(\pi_*X,\pi_*Y)+({\theta(X)},{\theta(Y)})=g(\pi_*X,\pi_*Y).\)
	\end{proof}
	
	Let $A$ be a connection on $P$. Recall that it induces an almost complex structure $J_A$ on $P$. Therefore, it is natural to determine if $(P,J_A,g_A)$ is a Hermitian manifold. Indeed, this is the case as shown below under appropriate assumptions on $M$ and $H$.
	
	\begin{proposition}
		Let \((M, j, g)\) be a Hermitian manifold, and \(H\) a complex Lie group equipped with a Hermitian structure \((H, i_H, (\cdot,\cdot))\). Let \(P\) be a principal \(H\)-bundle over \(M\), and \(A\) a connection on \(P\). Denote by \(J_A\) the complexion induced by \(A\), and by \(g_A\) the metric induced on \(P\). Then \((P, J_A, g_A)\) is a Hermitian manifold. In other words,
		\[
		g_A(J_A(\cdot), J_A(\cdot)) = g_A(\cdot, \cdot).
		\]
		
	\end{proposition}
	
	\begin{proof}
		The proof is split into three cases.
		\begin{itemize}
			\item{Case 1: Let $X=\eta_1^\#, Y=\eta_2^\#$.} By direct computations, we see that
			\begin{align*}
				g_A(J_A(\eta_1^\#),,J_A(\eta_2^\#))&=g_A((i_H\eta)^\#,(i_H\eta_2)^\#)\\
				&=(i_H\eta_1,i_H\eta_2)\\
				&=(\eta_1,\eta_2)\\
				&=g_A(\eta_1^\#,\eta_2^\#).
			\end{align*} 
			\item{Case 2: Let $X$ be vertical and $Y$ horizontal.}
			Recall that $J_A$ maps $VP$ to $VP$ and $A$ to $A$. Therefore, 
			\begin{align*}
				g_A(J_AX,J_AY)&=g({j(\pi_*X)},j(\pi_*Y))+(\theta(J_AX),{\theta(J_AY)})\\
				&=0\\
				&=g_A(X,Y).
			\end{align*}
			\item{Case 3: Suppose $X$ and $Y$ are horizontal.} It follows that
			\begin{align*}
				g_A(J_AX,J_AY)&=g(j(\pi_*X),j(\pi_*Y))+{(\theta(J_AX),{\theta(J_AY)})}\\
				&=g(\pi_*X,\pi_*Y) && ((M,j,g))\text{ is Hermitian})\\
				&=g_A(X,Y).
			\end{align*}
		\end{itemize}
	\end{proof}

	\subsection{K\"ahler Metrics Induced by Connections}
	
	In this section, we aim to prove that if $(M,j,g)$ is K\"ahler and $H$ is a complex Lie group admitting a Hermitian metric, then $(P,J_A,g_A)$ is K\"ahler if and only if the connection $A$ is flat.
	
	Assume $P$ is a principal $H$-bundle over a Hermitian manifold $(M,j,g)$ admitting a connection $A$. Let $g_A$ be the metric induced by $A$.
	Recall that $(P,J_A,g_A)$ is a Hermitian manifold. Let $\nabla^A$ denote the Levi-Civita connection of $g_A$ and $\nabla$ denote the Levi-Civita connection of $g$. We first state a lemma, then establish the relations between the Levi-Civita connections.
	
	\begin{lemma}
		Let $A$ be a connection with connection one-form $\theta$ on a principal $H$-bundle $P$ over $M$. Then  for all $X,Y\in \Gamma(TM)$, we have
		\[[\tilde{X},\tilde{Y}]=\widetilde{[X,Y]}-(F_A(\tilde{X},\tilde{Y}))^\#.\]
	\end{lemma}
	
	\begin{proof}
		Let $X,Y\in \Gamma(TM)$. The horizontal component of $[\tilde{X},\tilde{Y}]$ is $\widetilde{[X,Y]}$ since $\pi_*(\widetilde{[X,Y]})=[X,Y]=\pi_*[\tilde{X},\tilde{Y}]$. To find its vertical component, we use Equation \eqref{horizontal component with respect to a connection}. Recall that 
		\[F_A(\tilde{X},\tilde{Y})=\mrd\theta(\tilde{X},\tilde{Y})=\tilde{X}(\theta(\tilde{Y}))-\tilde{Y}\theta(\tilde{X})-\theta([\tilde{X},\tilde{Y}])=-\theta([\tilde{X},\tilde{Y}]).\] Its vertical component is therefore given by
		\begin{align*}
			(\theta[\tilde{X},\tilde{Y}])^\#
			&=-(F_A(\tilde{X},\tilde{Y}))^\#.
		\end{align*}
	\end{proof}

	The following lemma describes the relationship between the Levi-Civita connections of $M$ and $P$, which plays an important role in the proof of the main theorem.
	
	\begin{lemma}[\cite{Arash}]
		Let $H$ be a Lie group equipped with a bi-invariant metric $(\cdot\,,\,\cdot)$
		Let $U$ be an open subset of $M$ and $e=(e_1,\dots,e_n)$ a $g$-orthonormal local frame for $TM|_U$. For all $X,Y\in \Gamma(TU)$ and $\eta_1,\eta_2\in \mfh$, the following formulae hold.
		\begin{enumerate}
			\item $\nabla_{\tilde{X}}^A \tilde{Y}=\widetilde{\nabla_X Y}-\displaystyle\frac{1}{2}(F_A(\tilde{X},\tilde{Y}))^\#$.
			\item $\nabla_{\eta^\#}^A \tilde{X}=\nabla_{\tilde{X}}^A\eta^\#=\displaystyle\frac{1}{2}\sum_{k=1}^n (F_A(\tilde{X},\tilde{e_k}),\eta)\,\tilde{e_k}$.
			\item $\nabla_{\eta_1^\#}^A\eta_2^\#=\displaystyle\frac{1}{2}[\eta_1,\eta_2]^\#$.
		\end{enumerate}
	\end{lemma}
	
	\begin{proof}
		The Levi-Civita connection $\nabla^A$ of $g_A$ is characterised by the equation
		\begin{align}
			2g_A(\nabla_U^A V,W)&=U(g_A(V,W))+V(g_A(W,U))-W(g_A(U,V))\nonumber\\
			&\quad\quad\quad-g_A(U,[V,W])+g_A(V,[W,U])+g_A(W,[U,V]),
			\label{levi civita connection formula}
		\end{align}
		for $U,V,W\in \Gamma(TP)$. 
		
		\begin{enumerate}
			\item Let \( U = \tilde{X} \), \( V = \tilde{Y} \), and \( W = \tilde{Z} \) be the horizontal lifts of the vector fields \( X, Y, Z \in \Gamma(TM) \). Using Equation \eqref{levi civita connection formula}, we get
			\begin{align*}
				&2g_A(\nabla_{\tilde{X}}^A \tilde{Y},\tilde{Z})\\
				&= \tilde{X}(g_A(\tilde{Y},\tilde{Z}))+\tilde{Y}(g_A(\tilde{Z},\tilde{X}))-\tilde{Z}(g_A(\tilde{X},\tilde{Y}))\\
				&\quad\quad\quad-g_A(\tilde{X},[\tilde{Y},\tilde{Z}])+g_A(\tilde{Y},[\tilde{Z},\tilde{X}])+g_A(\tilde{Z},[\tilde{X},\tilde{Y}])\\
				&=X(g(Y,Z))+Y(g(Z,X))-Z(g(X,Y))\\
				&\quad\quad\quad-g(X,[Y,Z])+g(Y,[Z,X])+g(Z,[X,Y])\\
				&=2g(\nabla_X Y,Z)\\
				&=2g_A(\widetilde{\nabla_X Y},\tilde{Z})
			\end{align*}
			Thus, the horizontal projection of \( \nabla^A_{\tilde{X}} \tilde{Y} \) is \( \widetilde{\nabla_X Y} \).
			
			Now, we find its vertical projection. Let \( U = \tilde{X} \), \( V = \tilde{Y} \), and \( W = \eta^\# \) for some $\eta\in \mfh$. Again, using Equation \eqref{levi civita connection formula}, we see that
			\begin{align*}
				2g_A(\nabla^A_{\tilde{X}} \tilde{Y}, \eta^\#) &= \tilde{X} g_A(\tilde{Y}, \eta^\#) + \tilde{Y} g_A(\eta^\#, \tilde{X}) - \eta^\# g_A(\tilde{X}, \tilde{Y}) \\
				&\quad - g_A(\tilde{X}, [\tilde{Y}, \eta^\#]) + g_A(\tilde{Y}, [\eta^\#, \tilde{X}]) + g_A(\eta^\#, [\tilde{X}, \tilde{Y}]) \\
				&= 2g_A(\eta^\#, [X,Y] - F_A(X,Y)^\#) \\
				&= -g_A(F_A(X,Y)^\#, \eta^\#).
			\end{align*}
			Thus, the vertical projection of \( \nabla^A_{\tilde{X}} \tilde{Y} \) is \( -\frac{1}{2} F_A(X,Y)^\# \).
		\end{enumerate}
		
		The proof of the remaining claims are similar.   	
	\end{proof}

	\begin{theorem}
		Let $(M,j,g)$ be a K\"ahler manifold and $H$ a complex Lie group equipped with a bi-invariant Hermitian metric $(\,,\,)$. Then $(P,J_A,g_A)$ is K\"ahler if and only if the connection $A$ is flat, i.e, $F_A\equiv0$.
		\label{kahler iff flat}
	\end{theorem}
	
	\begin{proof}
		Assume first $F_A=0$. By Corollary \ref{integrability condition and F_A}, $J_A$ is integrable. It suffices to show $\nabla^A J_A=0$. We will show $\nabla^A J|_U=0$ over a trivialising cover $U$ of $P$.
		\begin{itemize}
			\item  Case 1: Let $V_1=\eta_1^\#$ and $V_2=\eta_2^\#$. Then
			\begin{align*}
				J_A\left(\nabla_{\eta_1^\#}^A \eta_2^\#\right)
				&=\frac{1}{2}J_A([\eta_1,\eta_2]^\#)\\
				&=\frac{1}{2}(i_H[\eta_1,\eta_2])^\#\\
				&=\frac{1}{2}[\eta_1,i_H\eta_2]^\#\\
				&=\nabla_{\eta_1^\#}^A(i_H\eta_2)^\#\\
				&=\nabla_{\eta_1^\#}^A(J\eta_2^\#).
			\end{align*}
			\item Case 2: Let $V_1=\tilde{X}$ and $V_2=\eta^\#$. By the previous lemma, it follows that $\nabla_{\tilde{X}}^A\eta^\#=0$ as $F_A\equiv 0$ and $[\tilde{X},\eta^\#]=0$. Since $J(\eta^\#)=(i_H\eta)^\#$, we see that $J(\nabla_{\tilde{X}}^A\eta^\#)=0=\nabla_{\tilde{X}}^A(J\eta^\#)$. The same conclusion holds if $V_1$ is vertical and $V_2$ is a horizontal lift.
			\item Case 3: Let $V_1=\tilde{X}$ and $V_2=\tilde{Y}$. Since $F_A$ vanishes and $j$ is K\"ahler, we see that
			\begin{align*}
				J\left(\nabla_{\tilde{X}}^A\tilde{Y}\right)
				&=J(\widetilde{\nabla_X Y})\\
				&=\widetilde{j(\nabla_XY)}\\
				&=\widetilde{\nabla_X(j(Y))}\\
				&=\nabla_{\tilde{X}}^A\left(J\tilde{Y}\right).
			\end{align*}
		\end{itemize}
		Conversely, suppose that $(P,J_A,g_A)$ is K\"ahler. Since $\nabla_{\tilde X}^A\eta^\#=\nabla_{\eta^\#}^A\tilde{X}$, we obtain
		\[\nabla_{\tilde X}^A(i_H\eta)^\# =J\left(\nabla_{\tilde X}^A\eta^\#\right)=J\left(\nabla_{\eta^\#}^A\tilde{X}\right)=\nabla_{\eta^\#}^A(\widetilde{j(X)}).\]
		Using the previous lemma, this is equivalent to
		\[\sum_{k=1}^{2n}(F_A(\tilde{X},\tilde{e_k}),i_H\eta)\tilde{e_k}=\sum_{k=1}^{2n}(F_A(J(\tilde{X}),\tilde{e_k}),\eta)\tilde{e_k}.\]
		Since $i_H$ is Hermitian, this reduces to
		\begin{equation}
			F_A(J(\tilde{X}),\tilde{Y})=-i_H F_A(\tilde{X},\tilde{Y}),
			\label{first equation F_A}
		\end{equation}
		for all vector fields $X,Y$ on $M$. On the other hand, by the previous lemma,
		\[\nabla_{\tilde{X}}^A \tilde{Y}= \widetilde{\nabla_X Y}-\frac{1}{2}(F_A(\tilde{X},\tilde{Y}))^\#.\]
		The equation $J(\nabla_{\tilde{X}}^A \tilde{Y})=\nabla_{\tilde{X}}^A (J\tilde{Y})$ is equivalent to 
		\[\widetilde{j(\nabla_X Y)}-\frac{1}{2}(i_H F_A(\tilde{X},\tilde{Y}))^\# = \widetilde{\nabla_X j(Y)}-\frac{1}{2}(F_A(\tilde{X},J(\tilde{Y})))^\#.\]
		Since $(M,j,g)$ is K\"ahler, this is equivalent to
		\begin{equation}
			F_A(\tilde{X},J(\tilde{Y}))=i_H F_A(\tilde{X},\tilde{Y}).
			\label{second equation F_A}
		\end{equation}
		By interchanging the roles of $X$ and $Y$ in Equation \eqref{second equation F_A} and combining the result with Equation \eqref{first equation F_A}, it follows that
		\[F_A(\tilde{X},\tilde{Y})=0,\]
		for all vector fields $X,Y$ on $M$. Since $F_A$ is horizontal, this implies that $F_A\equiv 0$.
	\end{proof}

	\subsection{K\"ahler Metrics on Homogeneous Bundles}
	In this section, we establish a one-to-one correspondence between the set of K\"ahler metrics induced by invariant connections on homogeneous bundles and a certain set of linear maps between Lie algebras, which, in fact, coincides with the set of Lie algebra homomorphisms between Lie algebras. This correspondence is an immediate consequence of Wang's classification theorem of invariant connections (Theorem \ref{Wang classification theorem}) and Theorem \ref{kahler iff flat}.
	
	\begin{theorem}
		Let $(M,j,g)$ be a K\"ahler and homogeneous manifold, and $H$ a complex Lie group equipped with a Hermitian metric $(\,,\,)$. Let $P=G\times_\lambda H$ where $\lambda:K\to H$ is a homomorphism. Then there exists a one-to-one correspondence between the set of K\"ahler metrics on $P$ induced by invariant connections and the subset of
		$\mathrm{Hom}(\mfg,\mfh)$ consisting of Lie algebra homomorphisms $\Lambda:\mfg\to\mfh$ such that
		\begin{enumerate}
			\item $\Lambda|_\mfk=\lambda_{*,e}$.
			\item $\Lambda\circ \Ad_k = \Ad_{\lambda(k)}\circ \Lambda$ for all $k\in K$.
		\end{enumerate} 
		\label{classification of kahler metrics (Gx_lambda H)}
	\end{theorem}
	
	\begin{proof}
		It is obvious that if $A$ and $A'$ are two connections on $P$, then their induced metrics on $P$ are different. By the Wang's classification theorem (Theorem \ref{Wang classification theorem}), it suffices to show that the linear map $\Lambda:\mfg\to\mfh$ corresponding to an invariant connection $A$ on $P$ is a Lie algebra homomorphism if and only if $A$ is flat. Let $A$ be an invariant connection on $P$ and $F_A$ its curvature 2-form. By Theorem \ref{wang classification theorem easier direction}, upon fixing $x_0\in M$ and $p_0\in P_{x_0}$, the curvature form $F_A$ satisfies
		\[(F_A)_{p_0}(\xi_{p_0}^*,\zeta_{p_0}^*)=[\Lambda(\xi),\Lambda(\zeta)]-\Lambda([\xi,\zeta]),\]
		where $\xi,\zeta\in \mfg$, and $\xi^*$ is the fundamental vector field on $P$ associated to $\xi$ via the action of $G$ on $P$. Therefore, $A$ is flat if and only if $F_A$ is a Lie algebra homomorphism.
	\end{proof}

	\newpage
	
	\section{Applications: The Upper Half Plane and Complex Projective Spaces }
	In this section, we apply the correspondence theorems (Theorems \ref{classification of invariant connections (Gx_lambda H)} and \ref{classification of integrable invariant complexions}) to the reduced frame bundle of the upper half plane and complex projective spaces. In the first example, direct computations will be adopted since the dimension is small. For the complex projective spaces, we will apply some basic knowledge from representation theory.

	\subsection{An Example: The Upper Half Plane}
	
	\subsubsection*{The Homogeneous Setup}
	We begin by describing the upper half plane as a homogeneous space. Let $\mbh$ denote the upper half plane:
	\[M=\mbh=\{z\in \mbc: \operatorname{Im}(z)>0\}.\]
	Let $\mathrm{PSL}(2,\mbr)$ be the projective special linear group:
	\[\mathrm{PSL}(2,\mbr)=\mathrm{SL}(2,\mbr)/\{\pm \mathrm{I}\}.\]
	Hence, in $\mathrm{PSL}(2,\mbr)$, every matrix is identified with its negative. Moreover, we have a homomorphism from $\mathrm{PSL}(2,\mbr)$ to the automorphism group of $\mbh$, by sending $\displaystyle \begin{pmatrix}
		a & b\\
		c& d\\
	\end{pmatrix}\in \mathrm{PSL}(2,\mbr)$ to the automorphism $f(z):=\displaystyle\frac{az+b}{cz+d}$. This action on $\mbh$ is transitive. Therefore, $\mbh$ is a homogeneous space for the Lie group $\mathrm{PSL}(2,\mbr)$. The stabiliser of the imaginary number $i\in \mbh$ under this action is $K=\mathrm{SO}(2)/\{\pm I\}$. Therefore, our Lie groups and algebras are
	\[
	G= \mathrm{PSL}(2,\mbr), \quad \mfg = \mathfrak{sl}(2,\mbr) = \{A\in \mathrm{M}_2(\mbr): \on{Tr}(A)= 0\},
	\]
	and
	\[
	K= \mathrm{SO}(2)/\{\pm I\},\quad \mfk = \mathfrak{so}(2) = \{A\in \mathrm{M}_2(\mbr): A^T=-A\}.
	\]
	Note that $K$ is connected.
	
	\subsubsection*{Invariant Connections on $G \times_\lambda H$}
	We first study the space of $G$-invariant connections on a principal $H$-bundle $P=G\times_\lambda H$ over $M=\mbh$. By Theorem \ref{classification of invariant connections (Gx_lambda H)}, it is in one-to-one correspondence with the set
	\[\left\{\Lambda\in \mr{Hom}\left(\frac{\mfg}{\mfk},\mfh\right) : \Lambda\circ \on{ad}_\kappa = \on{ad}_{\lambda(\kappa)}\circ \Lambda \quad \forall \kappa\in \mfk  \right\}.\]
	We choose the basis elements of $\mfg$ (traceless matrices) to be 
	\[
	\begin{pmatrix}
		1 & 0\\
		0 & -1
	\end{pmatrix}, 
	\begin{pmatrix}
		0 & 1\\
		1 & 0
	\end{pmatrix},
	\begin{pmatrix}
		0 & 1\\
		-1 & 0
	\end{pmatrix},
	\]	
	and the basis element of $\mfk$ (skew-symmetric matrices) to be 
	\[
	\kappa_0=
	\begin{pmatrix}
		0 & 1\\
		-1 & 0
	\end{pmatrix}.
	\]
	Therefore, the basis elements of $\mfg/\mfk$ are
	\[
	\begin{pmatrix}
		1 & 0\\
		0 & -1
	\end{pmatrix}, 
	\begin{pmatrix}
		0 & 1\\
		1 & 0
	\end{pmatrix}.
	\]
	Let $B=\Lambda\begin{pmatrix} 1 & 0\\ 0 & -1 \end{pmatrix}$ and $C=\Lambda\begin{pmatrix} 0 & 1\\ 1 & 0 \end{pmatrix}$.
	Now, $\Lambda\circ \on{ad}_\kappa=\operatorname{ad}_{\lambda(\kappa)}\circ \Lambda$ for all $\kappa\in \mfk$ if and only 
	\[
	\left[\lambda(\kappa_0),B\right]=\Lambda\left[\kappa_0,\begin{pmatrix} 1 & 0\\ 0 & -1 \end{pmatrix}\right]=-2C,
	\]
	and
	\[
	\left[\lambda(\kappa_0),C\right]=\Lambda\left[\kappa_0,\begin{pmatrix} 0 & 1\\ 1 & 0 \end{pmatrix}\right]=2B,
	\]
	By writing $A=B+iC$ and using the complex bilinear extension of the Lie bracket, the condition above is equivalent to $[\lambda(\kappa_0),A]=2iA$. Since $\Lambda$ is linear, it is uniquely determined by $B$ and $C$. Thus, the space of ${G}$-invariant connections on a principal $H$-bundle $P$ over $M=\mbh$ is in one-to-one correspondence with the set
	\[\{A\in \mfh_\mbc : [\lambda(\kappa_0),A]=2iA\}.\]
	
	\subsubsection*{Invariant Complexions on $G \times_\lambda H$}
	Now we study the space of ${G}$-invariant complexions on $P$ over $M=\mbh$. We assume our structure group $H$ is a complex Lie group with canonical almost complex structure $i_H$, and the base $M$ is endowed with the canonical almost complex structure inherited from $\mbc$. By Theorem \ref{classification of invariant complexions: G times_lambda H}, the space of ${G}$-invariant complexions on a principal $H$-bundle  $P$ over $M=\mbh$ is in one-to-one correspondence with the set
	\[\left\{\Lambda\in \mr{Hom}\left(\frac{\mfg_\mbc}{\mfk_\mbc},\mfh\right) : \Lambda\circ \on{ad}_\kappa = \on{ad}_{\lambda(\kappa)}\circ \Lambda \quad \forall \kappa\in \mfk,\quad \Lambda\circ j_\mfg = -i \Lambda  \right\},\]
	where $j_\mfg:\mfg/\mfk\to \mfg/\mfk$ is the almost complex structure $r_{x_0*}^{-1}\circ j_{x_0}\circ r_{x_0*}$ on $\mfg/\mfk$. In this case, $x_0=i$ and $j_{x_0}:T_{x_0}\mbh\to T_{x_0}\mbh$ is the canonical almost complex structure on $\mbh$. First, we find the map 
	\[
	r_{x_0*,e}:\frac{\mfg}{\mfk}=\mathrm{span}\left\{
	\begin{pmatrix}
		1 & 0\\
		0 & -1
	\end{pmatrix}, 
	\begin{pmatrix}
		0 & 1\\
		1 & 0
	\end{pmatrix}\right\} \to T_{x_0}\mbh.
	\]
	We compute that 
	\begin{align*}
		r_{x_0*,e}
		\begin{pmatrix}
			1 & 0\\ 0 & -1
		\end{pmatrix}
		&=\frac{\mrd}{\mrd t}\Bigg|_{t=0} \exp\left(t\begin{pmatrix} 1 & 0\\ 0 & -1 \end{pmatrix}\right)\cdot i\\
		&= \frac{\mrd}{\mrd t}\Bigg|_{t=0} \frac{(\exp t)i }{\exp(-t)}\\
		&=2\frac{\del}{\del y}.
	\end{align*}
	Similarly, one may compute that 
	\[r_{x_0*,e}\begin{pmatrix} 0 & 1\\1 & 0\end{pmatrix}=2\frac{\del}{\del x}.\]
	Therefore, the map $j_\mfg:\mfg/\mfk\to \mfg/\mfk$ is given by the following:
	\[
	j_\mfg\begin{pmatrix} 1&0\\0&-1 \end{pmatrix} = r_{x_0*}^{-1}\circ j_{x_0}\left(2\frac{\del}{\del y}\right)=-r_{x_0*}^{-1}\left(2\frac{\del}{\del x}\right)=-\begin{pmatrix} 0&1\\1&0 \end{pmatrix},
	\]
	\[
	j_\mfg\begin{pmatrix} 0&1\\1&0 \end{pmatrix} = r_{x_0*}^{-1}\circ j_{x_0}\left(2\frac{\del}{\del x}\right)=r_{x_0*}^{-1}\left(2\frac{\del}{\del y}\right)=\begin{pmatrix} 1&0\\0&-1 \end{pmatrix},
	\]
	
	Let $B=\Lambda\begin{pmatrix} 1 & 0\\ 0 & -1 \end{pmatrix}$ and $C=\Lambda\begin{pmatrix} 0 & 1\\ 1 & 0 \end{pmatrix}$. It follows that $\Lambda\circ j_\mfg=-i\Lambda$ if and only if 
	\[
	C = i_H B, \quad  B = -i_HC,
	\]
	which is equivalent to $C=i_H B$. As shown previously, the condition $\Lambda\circ \on{ad}_\kappa = \on{ad}_{\lambda_*(\kappa)}\circ \Lambda$ for all $\kappa\in \mfk$ is equivalent to 
	\[[\lambda_*(\kappa_0),B]=-2C,\quad [\lambda_*(\kappa_0),C]=2B.\]
	Therefore, the space of ${G}$-invariant complexions on a principal $H$-bundle $P$ over $M=\mbh$ is in one-to-one correspondence with the set
	\[\{B\in \mfh : [\lambda_*(\kappa_0),B]=-2i_H B\}.\]
	
	\subsubsection*{The Reduced Frame Bundle of the Upper Half Plane}
	We now apply the previous results to the frame bundle of $\mbh$. Let $P$ be the reduced frame bundle of $\mbh$ with structure group $\mathrm{GL}(1,\mbc)=\mbc^*$. This is possible because $\mbh$ admits an almost complex structure. Note that $\mbh$ is a symmetric space. To see this, let $\sigma:\mathrm{PSL}(2,\mbr)\to \mathrm{PSL}(2,\mbr)$ be the automorphism defined by
	\[\sigma(A):=(A^T)^{-1},\quad A\in \mathrm{PSL}(2,\mbr).\]
	It follows that 
	\[G_\sigma=\{A\in \mathrm{PSL}(2,\mbr): \sigma(A)=A\}=\{A\in\mathrm{PSL}(2,\mbr):A^T A=\mathrm{Id}\}=\mathrm{SO}(2).\]
	Since $G_\sigma=\mathrm{SO}(2)=K$ is connected, $K$ lies between $G_\sigma$ and the identity component of $G_\sigma$. Thus, $\mbh=G/K$ is a symmetric space (cf. Definition \ref{definition: symmetric spaces}). From our earlier results, the space of invariant complexions on $P$ is in one-to-one correspondence with the set
	\[\{z\in \mbc^*: [\lambda_*(\kappa_0),z]=-2iz \},\]
	where $\kappa_0$ is chosen appropriately. Since $\mbc^*$ is abelian, there is a unique invariant complexion on $P$. Moreover, by Corollary \ref{abelian structure group integrable}, this invariant complexion is integrable. Collecting these results, we conclude the following proposition.
	
	\begin{proposition}
		There is a unique integrable invariant complexion on the reduced frame bundle of the upper half plane $\mbh$.
		\label{unique complexion on reduced frame bundle of H}
	\end{proposition}

	\subsection{An Example: Complex Projective Spaces}
	
	\subsubsection*{The Homogeneous Setup}
	We now consider the more interesting case where $M=\mbcp^n$ for $n\geq1$. Recall that the complex projective space $\mbcp^n$ is defined as
	\[\mbcp^n=\{[z_0:\cdots:z_n]:(z_0,\cdots,z_n)\in \mbc^{n+1}-0\},\]
	where $[z_0:\cdots:z_n]=[\lambda z_0:\cdots:\lambda z_n]$ for all $\lambda\in \mbc$.
	
	Let $G$ be the group $SU(n+1)$ of all $(n+1)\times(n+1)$ special unitary matrices, namely, $U\in SU(n+1)$ if and only if $\det U=1$ and $UU^*=I$. The group $SU(n+1)$ acts on $\mbcp^n$ since matrices are linear. More precisely, for $U\in G, [z]\in \mbcp^n$, the action is defined as $U\cdot [z]:=[Uz].$ This action is transitive, making $\mbcp^n$ a homogeneous space. Arguing as in the case of the upper half plane, it follows that $\mbcp^n$ is also a symmetric space. Moreover, since elements of $SU(n+1)$ are complex linear, they are holomorphic and thus the complex structure $j_{\mbcp^n}$ is $SU(n+1)$-invariant.
	
	Choose $[e_0]\in\mbcp^n$ to be the base point of our action, where $e_0=(1,0,\cdots,0)\in \mbc^{n+1}$. The stabiliser $K$ of $[e_0]$ under this action is then
	\[K=S(U(n)\times U(1))=\left\{
	\begin{pmatrix}
		\lambda & 0\\
		0 & A
	\end{pmatrix}: |\lambda|=1, A\in U(n), \lambda\det A=1\right\}.\]
	
	\subsubsection*{Lie Algebraic Data}
	Let $P$ be the reduced frame bundle of $\mbcp^n$, which has the structure group $H=GL(n,\mbc)$. Hence, the Lie groups, Lie algebras involved and their complexifications are 
	\begin{itemize}
		\item $G=SU(n+1),$ \\
		$\mfg=\mathfrak{su}(n+1)=\{A\in M_n(\mbc): \on{Tr}(A)=0, A^*=-A\}$,\\
		$\mfg_\mbc=\mathfrak{sl}(n+1,\mbc)=\{A\in M_n(\mbc):\on{Tr}(A)=0\}$,
		\item $K=S(U(1)\times U(n))$,\\
		$\mfk=\left\{\begin{pmatrix}
			ia & 0\\
			0 & A
		\end{pmatrix}: a\in \mbr, A\in \mathfrak{u}(n), ia+\on{Tr}(A)=0\right\}$,\\
		$\mfk_\mbc=\left\{\begin{pmatrix}
			z & 0\\
			0 & A
		\end{pmatrix}: z\in \mbc, A\in M_n(\mbc), z+\on{Tr}(A)=0\right\}$
		\item $H=GL(n,\mbc)$,\\
		$\mfh=\mathfrak{gl}(n,\mbc)=M_n(\mbc)$.
	\end{itemize}
	Note that $\mfh$ is complexified using the canonical complex structure, i.e, multiplying by the imaginary number $i$. Moreover, an $\Ad_K$-invariant splitting of $\mfg$ is given by
	\[\mfg=\mfk\oplus\mfm,\]
	where $\mfm=\left\{
	\begin{pmatrix}
		0 &-\bar{z}^T\\
		z & 0
	\end{pmatrix}
	:z\in \mbc^n\right\}$ with complexification $\mfm_\mbc=\left\{
	\begin{pmatrix}
		0 &{w}^T\\
		z & 0
	\end{pmatrix}
	:z,w\in \mbc^n\right\}$. It is in fact the orthogonal complement of $\mfk$ with respect to the inner product $\langle U,V\rangle=-\on{Tr}(UV)$ on $SU(n+1)$.
	
	\subsubsection*{The Complex Structure on $\mathfrak{m}$}
	Recall that the almost complex structure $j_\mfm:\mfm\to\mfm$ is defined via the commutative diagram
	
	\adjustbox{scale=1.3,center}{
		\begin{tikzcd}
			T_{x_0}M \arrow[r, "j_{x_0}"]                                               & T_{x_0}M                                 \\
			\mathfrak{m} \arrow[r, "j_{\mathfrak{m}}", dashed] \arrow[u, "{r_{x_0 *,e}}"] & \mathfrak{m} \arrow[u, "{r_{x_0 *,e}}"']
		\end{tikzcd}
	}
	where $x_0=[e_0]$ in our case. Now we compute $r_{x_0*,e}:\mfm\to T_{x_0}M$. Let 
	\mbox{$\begin{pmatrix}
			0 &-\bar{z}^T\\
			z & 0
		\end{pmatrix}\in \mfm$}. Then
	\begin{align*}
		r_{x_0*}
		\begin{pmatrix}
			0 &-\bar{z}^T\\
			z & 0
		\end{pmatrix}
		&=\frac{\mrd}{\mrd t}\Big|_{t=0}\left(I_n+
		\begin{pmatrix}
			0 &-\bar{z}^T\\
			z & 0
		\end{pmatrix}+O(t^2)\right)[e_0]\\
		&=\frac{\mrd}{\mrd t}\Big|_{t=0}[1:tz].
	\end{align*}
	Using the affine coordinates $[1:w_1:\cdots:w_n]\leftrightarrow(w_1,\cdots,w_n)$ at $[e_0]$, it follows that
	\[r_{x_0*}
	\begin{pmatrix}
		0 &-\bar{z}^T\\
		z & 0
	\end{pmatrix}=\frac{\mrd}{\mrd t}\Big|_{t=0}tz = z.
	\]
	Hence, if a tangent vector of $T_{[e_0]}\mbcp^n$ is locally represented by $z$, then
	\begin{align*}
		j_\mfm
		\begin{pmatrix}
			0 &-\bar{z}^T\\
			z & 0
		\end{pmatrix}=
		\begin{pmatrix}
			0 &i\bar{z}^T\\
			iz & 0
		\end{pmatrix}.
	\end{align*}
	Its complexification $j_\mfm:\mfm_\mbc\to\mfm_\mbc$ is 
	\begin{align*}
		j_\mfm
		\begin{pmatrix}
			0 &w^T\\
			z & 0
		\end{pmatrix}=
		\begin{pmatrix}
			0 & -i w^T\\
			iz & 0
		\end{pmatrix}, \quad z,w\in \mbc^n.
	\end{align*}
	From here, it is easy to see that the holomorphic and the anti-holomorphic tangent spaces of $\mfm$ are respectively given by
	\begin{align*}
		\mfm^{1,0}&= \left\{
		\begin{pmatrix}
			0 & 0\\
			z & 0
		\end{pmatrix}:
		z\in \mbc^n
		\right\},\\
		\mfm^{0,1}&= \left\{
		\begin{pmatrix}
			0 & z^T\\
			0 & 0
		\end{pmatrix}:
		z\in \mbc^n
		\right\}.
	\end{align*}
	
	\subsubsection*{The Intertwining Operator}
	Finally, by Theorem \ref{classification of integrable invariant complexions}, the space of integrable invariant complexions on the reduced frame bundle $P$ of $\mbcp^n$ are in one-to-one correspondence with the subset of $\mathrm{Hom}(\mfg_\mbc,\mfh)$ consisting of linear maps 
	$\psi:\mfg_\mbc\to (\mfh,i_H)$ such that \[\psi|_{\mfk_\mbc}=0, \quad \psi\circ \on{ad}_\kappa=\on{ad}_{\lambda_*(\kappa)}\circ \psi \quad\forall \kappa\in \mfk_\mbc,\quad\psi\circ j_\mfm=-i\psi, \quad [\psi,\psi]=0.\]
	The first and the third conditions imply that $\psi$ is a map from $\mfm^{0,1}$ to $\mfh$. We claim that $\psi$ has to be zero using the second condition. Since the zero map satisfies the fourth condition, this would then conclude the uniqueness of invariant integrable complexion on the reduced frame bundle of $\mbcp^n$.
	
	Note that since $j_{\mbcp^n}$ is $G$-invariant, the Lie algebra $\mfk_\mbc$ acts on $\mfm^{0,1}$ via the restricted adjoint action. More precisely, let $\rho_1:\mfk_\mbc\to \mathfrak{gl}(m^{0,1})$ be the adjoint representation of $\mfk_\mbc$ restricted to $\mfm^{0,1}$:
	\[ \rho_1(\kappa)(\xi)=[\kappa,\xi],\quad\kappa\in \mfk_\mbc, \xi\in \mfm^{0,1}.\]
	Moreover, another representation of $\mfk_\mbc$ on $\mfh$ is given by $\rho_2:\mfk_\mbc\to\mathfrak{gl}(\mfh)$
	\[\rho_2(\kappa)(\eta)=[\lambda_* \kappa, \eta],\quad \kappa\in \mfk_\mbc, \eta\in \mfh.\]
	Therefore, the second condition is equivalent to saying the map $\psi:\mfm^{0,1}\to\mfh$ is an intertwining operator between the representations $\rho_1$ and $\rho_2$.
	
	\begin{remark}
		Note that the formulation of $\psi:\mfm^{0,1}\to\mfh$ as an intertwining operator between $\rho_1$ and $\rho_2$ relies only on the $G$-invariance of the underlying complex structure, and therefore applies to invariant complexions on general homogeneous principal bundles that satisfy the assumption of Theorem \ref{classification of integrable invariant complexions}.
	\end{remark}
	
	The following lemma from representation theory is not required for the proofs of our results. Nevertheless, it motivates the general approach for solving related problems.
	
	\begin{lemma}
		Let $\rho:\mfk\to \mathfrak{gl}(V)$ be an irreducible complex representation of a Lie algebra $\mfk$. Let $Z(\mfk)=\{Z\in\mfk: [Z,\kappa]=0 \text{ for all }\kappa\in \mfk\}$ be the center of $\mfk$. For every $Z\in Z(\mfk)$, there exists a complex number $\lambda=\lambda_Z$ such that $\rho(Z)=\lambda\, id_V$.
		\label{center: representations are scalars}
	\end{lemma}
	
	\begin{proof}
		Let $\lambda$ be an eigenvalue of $\rho(Z)$ and $E_\lambda$ be the eigenspace of $\rho(Z)$ with eigenvalue $\lambda$. It follows that $E_\lambda$ is invariant under $\rho(Z)$ since $\rho(Z)$ is a Lie algebra homomorphism and $Z$ commutes with every element of $\mfk$. Since the representation $V$ is irreducible and $E_\lambda$ has a non-zero eigenvector of $\rho(Z)$, we conclude that $E_\lambda=V$. In other words, $\rho(Z)=\lambda \,id_V$.
	\end{proof}
	
	The following lemma is key to our results.
	
	\begin{lemma}
		Let $\rho_1:\mfk\to\mathfrak{gl}(V)$ and $\rho_2:\mfk\to\mathfrak{gl}(W)$ be complex representations of a Lie algebra $\mfk$. Let $Z\in\mfk$ be a non-zero vector. Suppose there exists $\lambda_V,\lambda_W\in\mbc$ such that 
		\[\rho_1(Z)=\lambda_V\,id_V,\quad \rho_2(Z)=\lambda_W\,id_W.\]
		If $\lambda_V\neq\lambda_W$, then every intertwining operator $\psi:V\to W$ must be zero.
		\label{psi=0}
	\end{lemma}
	
	\begin{proof}
		Let $v\in V$. By the assumptions and the intertwining property of $\psi$, we see that 
		\[\lambda_V \psi(v)=\psi(\rho_1(Z)v)=\rho_2(Z)\psi(v)=\lambda_W\psi(v).\]
		Since $\lambda_V\neq\lambda_W$, the map $\psi$ has to be zero.
	\end{proof}
	
	\subsubsection*{The Two Representations of $\mfk_\mbc$}
	We now compute the representations $\rho_1$ and $\rho_2$ of $\mfk_\mbc=\mathfrak{s}(\mathfrak{u}(1)\oplus \mathfrak{u}(n))$. Let \mbox{$\kappa=
	\begin{pmatrix}
		z & 0\\
		0 & A
	\end{pmatrix}$} be an element of $\mfk_\mbc$, where $z\in \mbc, A\in M_n(\mbc)$ and $a+\on{Tr}(A)=0$.

	To compute the first representation $\rho_1:\mfk_\mbc\to\mathfrak{gl}(\mfm^{0,1})$, let $\xi=\begin{pmatrix}
		0&u^T\\
		0 & 0
	\end{pmatrix}\in \mfm^{0,1}$. Then $\rho_1$ is given by
	\begin{align*}
		\rho_1(\kappa)(\xi)
		&=[\kappa,\xi]\\
		&=
		\begin{pmatrix}
			0 & u^T(aI_n-A)\\
			0 & 0
		\end{pmatrix}.
	\end{align*}
	 
	 Before computing the second representation $\rho_2:\mfk_\mbc\to\mathfrak{gl}(\mfh)=\mathfrak{gl}(M_n(\mbc)))$, we need to find the map $\lambda_{*,e}:\mfk\to\mfh$. Recall that an element of the frame bundle $P$ can be viewed as a complex linear isomorphism from $\mbc^n$ to $T_{x_0}M$, and the action of $H=GL(n,\mbc)$ on $P$ is composition of maps from the right. We define a canonical identification $\iota: \mbc^n \xrightarrow{\sim} \mfm$ by 
	\[\iota(z) = \begin{pmatrix} 0 & -\bar{z}^T \\ z & 0 \end{pmatrix},\]
	and set our reference frame to $p_0 := r_{x_0*, e} \circ \iota\in P_{x_0}$. 
	
	By the definition of the isotropy representation $\lambda: K \to H$, the left action of $k \in K$ on $p_0$ equals the right principal action by $\lambda(k)$:
	\[L_{k*,x_0} \circ p_0 = p_0 \circ \lambda(k).\]
	Since $L_k(r_{x_0}(\exp(tX))) = r_{x_0}(\exp(t \on{Ad}_k X))$ for $X \in \mfm$, taking the derivative at $t=0$ yields $L_{k*, x_0} \circ r_{x_0*, e} = r_{x_0*, e} \circ \on{Ad}_k|_\mfm$. Substituting this into the frame equation yields $\lambda(k) = \iota^{-1} \circ \on{Ad}_k|_\mfm \circ \iota$.  The induced Lie algebra homomorphism $\lambda_*: \mfk \to \mfh$ is hence:
	\[\lambda_*(\kappa) = \iota^{-1} \circ \on{ad}_\kappa|_\mfm \circ \iota, \quad \kappa \in \mfk.\]
	
	We compute this explicitly for $\mbcp^n$. Let $\kappa = \begin{pmatrix} ia & 0 \\ 0 & A \end{pmatrix} \in \mfk$ and $X = \iota(z) \in \mfm$. The matrix commutator is:
	\begin{align*}
		[\kappa, \iota(z)] &= \begin{pmatrix} ia & 0 \\ 0 & A \end{pmatrix} \begin{pmatrix} 0 & -\bar{z}^T \\ z & 0 \end{pmatrix} - \begin{pmatrix} 0 & -\bar{z}^T \\ z & 0 \end{pmatrix} \begin{pmatrix} ia & 0 \\ 0 & A \end{pmatrix} \\
		&= \begin{pmatrix} 0 & -\bar{z}^T A - ia\bar{z}^T \\ (A - iaI_n)z & 0 \end{pmatrix}.
	\end{align*}
	Since $A \in \mathfrak{u}(n)$, we have $A^* = -A$, which implies $\bar{z}^T A = -\overline{(Az)}^T$. Thus, the resulting matrix is exactly $\iota((A - iaI_n)z)$. Consequently, the map $\lambda_*: \mfk \to \mathfrak{gl}(n,\mbc)$ is given by:
	\[\lambda_*\begin{pmatrix} ia & 0 \\ 0 & A \end{pmatrix} = A - iaI_n.\]
	
	Consequently, the representation $\rho_2:\mfk_\mbc\to\mathfrak{gl}(M_n(\mbc))$ is
	\[
		\rho_2\begin{pmatrix} z & 0 \\ 0 & A\end{pmatrix}(B)=[A-zI_n,B]=\on{ad}_A(B), \quad B\in M_n(C).
	\]
	
	\subsubsection*{Uniqueness of Invariant Integrable Complexions on $P$}
	Let $Z=\begin{pmatrix}
		-n & 0\\
		0 & I_n
		\end{pmatrix}\in \mfk_\mbc$. We will show that $\rho_1(Z)$ and $\rho_2(Z)$ both act by scalar multiplication.
		
		Let $\begin{pmatrix}
			0 & u^T\\
			0 & 0
		\end{pmatrix}\in \mfm^{0,1}$. Then 
		\[\rho_1(Z)
		\begin{pmatrix}
			0 & u^T\\
			0 & 0
		\end{pmatrix}
		=-(n+1)
		\begin{pmatrix}
			0 & u^T\\
			0 & 0
		\end{pmatrix}.
		\]
	For $B\in M_n(\mbc)$, 
	\[\rho_2(Z)(B)=\on{ad}_{I_n}(B)=0.\]
	
	\begin{remark}
		The element $Z$ is in fact in the center $Z(\mfk_\mbc)$ of $\mfk_\mbc$ (cf. Lemma \ref{center: representations are scalars}).
	\end{remark}
	
	Therefore, by Lemma \ref{psi=0}, any intertwining operator $\psi:\mfm^{1,0}\to\mfh$ has to be zero. This concludes the uniqueness of invariant integrable complexions on $P$.
	
	\begin{proposition}
		There is a unique invariant integrable complexion on the reduced frame bundle of $\mbcp^n$.
		\label{unique complexion on reduced frame bundle of CP^n}
	\end{proposition}

\bibliographystyle{alpha}
\bibliography{refs}

\end{document}